\documentclass[10pt]{article}
\usepackage[francais,english]{babel}
\usepackage{amsmath}
\usepackage{amsfonts}
\usepackage{amssymb}
\usepackage{amsthm}
\usepackage{graphics}
\usepackage{amscd}
\usepackage{fullpage}
\usepackage{epsfig}
\usepackage{color}
\newcommand{\R}{\mathbb{R}}
\newcommand{\N}{\mathbb{N}}

\newcommand{\C}{\mathbb{C}}

\newcommand{\E}{\mathbb{E}}

\newcommand{\mc}{\mathcal}

\newcommand{\eps}{\varepsilon}
\newcommand{\ind}{{\bf 1}}
\renewcommand{\P}{\mathbb{P}}

\renewcommand{\H}{\mathbb{H}}

\DeclareMathOperator{\diam}{diam}
\DeclareMathOperator{\dist}{dist}

\DeclareMathOperator{\SLE}{SLE}

\DeclareMathOperator{\Beta}{Beta}

\title{Duality of Schramm-Loewner Evolutions}
\author{Julien Dub\'edat\footnote{Partially supported by NSF grant DMS0804314}}
\newtheorem{thm}{Theorem}

\newtheorem{Thm}[thm]{Theorem}

\newtheorem{Prop}[thm]{Proposition}
\newtheorem{Lem}[thm]{Lemma}
\newtheorem{Cor}[thm]{Corollary}

\newtheorem{Rem}[thm]{Remark}

\begin{document}
\maketitle
\begin{abstract}
In this note, we prove a version of the conjectured duality for Schramm-Loewner Evolutions, by establishing exact identities in distribution between some boundary arcs of chordal $\SLE_\kappa$, $\kappa>4$, and appropriate versions of $\SLE_{\hat\kappa}$, $\hat\kappa=16/\kappa$.
\end{abstract}

\section{Introduction}

Schramm-Loewner Evolutions (or $\SLE$), introduced by Schramm in 1999, are probability distributions, parameterized by $\kappa>0$ on non-self traversing curves (the trace) connecting two boundary points in a planar, simply connected domain. They are characterized by a conformal invariance condition and a domain Markov property. See \cite{Lawler,W1} for general SLE background.

The geometric properties of the trace vary with the parameter $\kappa$. In particular, when $\kappa\leq 4$, the trace is a.s. a simple curve; this is no longer the case if $\kappa>4$ (\cite{RS01}). The trace stopped at some finite time is then distinct from its boundary. The duality conjecture for $\SLE$, roughly stated, is that a boundary arc of $\SLE_\kappa$ is locally absolutely continuous w.r.t. to (some version)
of $\SLE_{\hat\kappa}$, $\hat\kappa=16/\kappa$. This was suggested by Duplantier. In the case $(\kappa,\hat\kappa)=(8,2)$, this follows from the exact combinatorial relation between Loop-Erased Random Walks and Uniform Spanning Trees and the identification of their scaling limits in terms of $\SLE$  (\cite{LSW2}). In the case $(\kappa,\hat\kappa)=(6,8/3)$, it follows from the locality/restriction framework (\cite{LSW3}). An approach based on a relation with the free field has been proposed by Sheffield.

A precise duality conjecture is stated in \cite{Dub4} and elaborated on in \cite{Dub6}; we prove slightly different versions here. These involves variants of $\SLE_\kappa$: the $\SLE_\kappa(\underline\rho)$ processes ($\underline\rho=\rho_1,\dots,\rho_n$). They satisfy a domain Markov property when keeping track of $n$ marked points $z_1,\dots,z_n$ (in addition of the origin and the target of chordal $\SLE$). The influence of $z_i$ on the $\SLE$ trace is quantified by the real parameter $\rho_i$; this influence is attractive for $\rho_i>0$ and repulsive for $\rho_i<0$.

Let us consider a chordal $\SLE$ in the upper half-plane $\H$, going from 0 to infinity. In the phase $4<\kappa<8$, a boundary point, say $1$, is ``swallowed", ie gets disconnected from infinity by the trace at a random time $\tau_1$ when the trace hits some point in $(1,\infty)$. The boundary arc straddling 1 is the boundary arc seen by 1 at time $\tau_1^-$.

\begin{Thm}\label{dual4}
Consider a chordal $\SLE_\kappa$ in $(\H,0,\infty)$, $4<\kappa<8$;  
let $D$ be the leftmost visited point on $(1,\infty)$. Conditionally on $D$, the boundary arc straddling 1 is distributed as an $\SLE_{\hat\kappa}(-\frac{\hat\kappa}2,\hat\kappa-4,\hat\kappa-2)$ in $(\H,D,\infty,0,1,D^+)$, stopped when it hits $(0,1)$.
\end{Thm} 

In the phase $\kappa\geq 8$, a.s. every point in $\overline\H$ is visited by the trace. We isolate a boundary arc in a different way. Let $G$ be the leftmost point on $(-\infty,0)$ visited by the trace before $\tau_1$. We consider the boundary of $K_{\tau_G}$, the hull of the $\SLE$ stopped when it first visits $G$; this boundary is an arc between $G$ and a point in $(0,1)$.

\begin{Thm}\label{dual8}
Consider a chordal $\SLE_\kappa$ in $(\H,0,\infty)$, $\kappa\geq8$. 
Let $G$ be the leftmost visited point $G$ on $(-\infty,0)$ before $\tau_1$. Conditionally on $G$, the boundary of $K_{\tau_G}$ is distributed as an $\SLE_{\hat\kappa}(\frac{\hat\kappa}2,\frac{\hat\kappa}2-2,-\frac{\hat\kappa}2,\hat\kappa-4)$ in $(\H,G,\infty,G^-,G^+,0,1)$, stopped when it hits $(0,1)$.
\end{Thm} 

The distributions of $D$ and $G$ are well known and easy to derive.

In \cite{Dub6}, it is shown that duality shares common features with reversibility and the question of defining multiple $\SLE$ strands in a common domain. This {\em local commutation} property states that two $\SLE$ strands can be grown in a domain to a positive size, in a way that does not depend on the order in which the $\SLE$'s are growing. Such systems of commuting $\SLE$'s are classified in \cite{Dub6}; in particular, two versions of $\SLE_\kappa$, $\SLE_{\hat\kappa}$ can commute only if $\hat\kappa\in\{\kappa,16/\kappa\}$.

In \cite{DZrevers}, Zhan proves reversibility of chordal $\SLE_\kappa$, $\kappa\leq 4$, i.e. that the range of the trace of an $\SLE_\kappa$ in $D$ going from $x$ to $y$ has the same distribution as the range of the trace  of $\SLE$ going from $y$ to $x$ in $D$. This was previously known for $\kappa\in\{2,8/3,4,6,8\}$. The argument involves a sequence of coupling of an $\SLE_\kappa(D,x,y)$ with an $\SLE_\kappa(D,y,x)$, such that each coupling in the sequence is absolutely continuous w.r.t. the trivial (independent) coupling, and the limiting coupling is exact (the ranges of the two traces are identical).

Let $\gamma$, $\hat\gamma$ be traces of two $\SLE$'s satisfying the local commutation condition. Then, for $U,V$ disjoint open subsets of the domain, one has a coupling of $(\gamma,\hat\gamma)$ which is ``correct" on the time set $\{(s,t): s\leq\tau, t\leq\hat\tau\}$, where $\tau$, $\hat\tau$ are stopping times for the two $\SLE$'s, such that $\gamma^\tau\subset U$, $\hat\gamma^{\hat\tau}\subset V$. We construct a coupling of $(\gamma,\hat\gamma)$, which is ``correct" on the time set $\{(s,t): \gamma_{[0,s]}\cap\hat\gamma_{[0,t]}=\varnothing\}$. See Theorem \ref{maxcoupl} for a precise statement.

The duality identities follow from applying Theorem \ref{maxcoupl} to appropriate pairs of commuting $\SLE$'s, together with some {\em a priori} geometric information on the traces. Plainly, many identities may be generated in this fashion.

The article is organized as follows. Section 2 recalls some absolute continuity properties of chordal $\SLE$. Local commutation is discussed in Section 3. Maximal couplings of commuting $\SLE$'s are constructed in Section 4. Geometric consequences (in particular duality) are drawn in Section 5. Some technical lemmas are postponed to Section 6. 

\noindent{\bf Acknowledgments.} I wish to thank Greg Lawler for comments on an earlier version of this article and Oded Schramm for useful conversations.

\section{Absolute continuity for chordal SLE}

In this section we consider some absolute continuity properties of chordal $\SLE$, mostly based on \cite{LSW3}.  
Chordal $\SLE$ will also serve as a reference measure for variants we will study later.

We adopt the following notation: $c=(D,x,y)$ is a configuration where $D$ is a simply connected domain and $x,y$ are distinct boundary points. Unless there is an ambiguity, the configuration is simply denoted by $D$. The chordal $\SLE_\kappa$ measure on $c=(D,x,y)$ is denoted by $\mu_c$ ($\kappa$ is fixed). It is seen as a measure on Loewner chains up to increasing time change; or as a configuration-valued continuous process (up to time change); or as a measure on non self-traversing paths (\cite{RS01}). This path (the SLE ``trace") is denoted by $\gamma$, while the hull it generates is denoted by $K$. Let $U$ be a subdomain of $D$, agreeing with $D$ in a neighbourhood of $x$, and not containing $y$ on its boundary. Then $\mu_D^{U}$  denotes the measure on paths induced by chordal $\SLE$ starting from $x$ and stopped on exiting $U$; this happens at a random time $\tau$, at which the hull is $K_\tau$, the tip of the trace is $\gamma_\tau$, and the configuration $c_\tau$ is $(D_\tau=D\setminus K_\tau,\gamma_\tau,y)$. More generally, for $\tau$ a stopping time, $\gamma^\tau$ denotes the trace stopped at $\tau$ (ie the process up to time $\tau$), $\mu_c^\tau$ the measure induced by stopping at $\tau$. We will use $\gamma$ to denote both the trace as a process and as a subset of $\overline D$ (the range of the process).

Later on, we will use tightness conditions, so we shall review some technical points now. Let $(D,x,y)$ be a (bounded) configuration, $K$ a hull such that $(D\setminus K,x',y)$ is a configuration for some $x'\in\partial K$. By the Riemann mapping theorem, there is a conformal equivalence $\phi_K:D\setminus K\rightarrow D$; one can specify it uniquely by requiring its 2-jet at $y$ to be trivial ($\phi_K(y)=y$, $\phi_K'(y)=1$, $\phi_K''(y)=0$ if $\phi_K$ extends smoothly at $y$; this condition is coordinate independent, so one can first ``straighten" the boundary at $y$). One defines a topology on hulls as follows: $(K_n)$ converges to $K$ if $\phi_{K_n}$ converges to $\phi_{K}$ uniformly on compact sets of $\overline D$ that are at positive distance of $K$. This is a version of Carath\'eodory convergence. A topology on chains $(K_t)_{t\geq 0}$ is given by the condition: $(K_t^n)_t$ converges to $(K_t)_t$ if for all $T>0$, $K'$ a compact set of $\overline D$ at positive distance of $K_T$, $\phi_{K^n_t}$ converges to $\phi_{K_t}$ uniformly on $[0,T]\times K'$. Then the Loewner equation maps continuously $C(\R^+,\R)$ (with the usual topology of uniform convergence on compact sets) to the space of chains endowed with this topology (\cite{Lawler} Section 4.7). Thus the induced measure on chains is a Radon measure. From \cite{RS01}, we know that the chain is a.s. generated by a continuous non self-traversing path $\gamma$. For clarity, we will think of $\SLE$ as a measure on such paths, with the topology on chains described above.

To express densities, we need to define some conformal invariants. Let $(D,x,y)$ be a configuration, $z_x,z_y$ analytic local coordinates at the boundary ($z_x$ mapping a neighbourhood of $x$ in $D$ to the neighbourhood of $0$ in the upper semidisk). The Poisson excursion kernel is defined as
$$H_D(x,y)=\lim_{X\rightarrow x,Y\rightarrow y}\frac{G_D(X,Y)}{\Im(z_x(X))\Im(z_y(Y))}$$ 
where $G_D$ is the Green function in $D$ (with Dirichlet boundary conditions); this depends on the choice of $z_x$ (or $z_y$) as a (real) 1-form. If $D$ and $D'$ agree in a neighbourhood of $x$, we choose the same local coordinate $z_x$, so that $H_{D'}(x,y')/H_D(x,y)$ does not depend on a choice of local coordinate at $x$. Similarly for $i,j=1,2$, consider configurations $(D_{ij},x_i,y_j)$ such that $D_{ij}$ agrees with $D_{i,3-j}$ in a neighbourhood of $x_i$ and with $D_{3-i,j}$ in a neighbourhood of $y_j$. Then the ratio:
$$\frac{H_{D_{11}}(x_1,y_1)H_{D_{22}}(x_2,y_2)}{H_{D_{12}}(x_1,y_2)H_{D_{21}}(x_2,y_1)}$$
is defined independently of any (coherent) choice of local coordinates. To simplify the notation, if $c=(D,x,y)$ is a configuration, we set $H(c)=H_D(x,y)$.

There is a $\sigma$-finite measure $\mu^{loop}$ on unrooted loops in $\C$, the Brownian loop measure (\cite{LSW3,LW}). As in \cite{KLconf}, let us denote
$$m(D;K,K')=\mu^{loop}\{\delta: \delta\subset D, \delta\cap K\neq\varnothing, \delta\cap K'\neq\varnothing\}.$$
In accordance with \cite{LSW3}, set $\alpha=\alpha_\kappa=\frac{6-\kappa}{2\kappa}$, $\lambda=\lambda_\kappa=\frac{(6-\kappa)(8-3\kappa)}{2\kappa}$.

\begin{Prop}\label{Pdens}
If $c=(D,x,y)$ and $c'=(D',x,y')$ are configurations agreeing in a neighbourhood $U$ of $x$, $\partial U\cap\partial D=\partial U\cap\partial D'$ a connected arc containing $x$ at positive distance of $y,y'$, then $\mu_{c}^U$ and $\mu_{c'}^U$ are mutually absolutely continuous, with density
$$\frac{d\mu_{c'}^U}{d\mu_{c}^U}(\gamma)=\left(\frac{H(c'_\tau)H(c)}{H(c_\tau)H(c')}\right)^\alpha\exp(-\lambda m(D;K_\tau,D\setminus D')+\lambda m(D';K_\tau,D'\setminus D)) $$
uniformly bounded above and below.
\end{Prop}
\begin{proof}
We will reduce the statement to two known cases.

1. Assume that $D=D'$. Then the statement follows from Lemma 3.2 in \cite{Dub6}; see also \cite{SchWil}. More precisely, consider the following situation: $\H$ is the upper half-plane, three boundary points  $x,y,y'$ are marked; $K$ is a hull around $x$, $x'$ its tip, $\phi$ a conformal equivalence $\H\setminus K\rightarrow \H$.
Then $H_\H(x,y)=(x-y)^{-2}$, $H_\H(x,y')=(x-y')^{-2}$, computing in the natural local coordinate. It follows that $H_{\H\setminus K}(x',y)=\frac{\phi'(y)}{(\phi(y)-\phi(x'))^2}$, $H_{\H\setminus K}(x',y')=\frac{\phi'(y')}{(\phi(y')-\phi(x'))^2}$, for an appropriate (common) local coordinate at $x'$. Then the ratio
$$\frac{H_{\H\setminus K}(x',y')H_\H(x,y)}{H_{\H\setminus K}(x',y)H_\H(x,y')}=\phi'(y')\left(\frac{y'-x'}{\phi(y')-\phi(x')}\right)^2\phi'(y)^{-1}\left(\frac{y'-x'}{\phi(y')-\phi(x')}\right)^{-2}$$
is independent of (coherent) choices. One concludes by identifying the density of an $\SLE_\kappa(\H,x,y)$ and an $\SLE_\kappa(\H,x,y')$ with respect to the common reference measure $\SLE_\kappa(\H,x,\infty)$.
\\ 
2. Assume that $D'\subset D$, $D'$ and $D$ agree in a neighbourhood of $y=y'$. Then the statement is a rephrasing of Proposition 5.3 in \cite{LSW3}.\\
3. The general case reduces to 1,2 as follows. Let $y''$ be a point on the connected boundary arc of $x$ in $\partial D\cap\partial D'$, which is not on $\overline{\partial U}$; and $V$ the connected component of $D\cap D'$ having $x$ on its boundary. Then apply 1 to go from $(D,x,y)$ to $(D,x,y'')$; then 2 to go from $(D,x,y'')$ to $(V,x,y'')$; then 1 to go from $(V,x,y'')$ to $(D',x,y'')$; then 2 to go from $(D',x,y'')$ to $(D',x,y')$. Cancellations occur due to the ``inclusion exclusion" form of the ratios $(H(c'_\tau)H(c)/H(c_\tau)H(c'))$. 

For a general bound on densities, see Lemma \ref{Lbound}.
\end{proof}

\section{Local commutation}

\subsection{Reversibility}

Following the discussion in Section 2.1 of \cite{Dub6}, we phrase and then check a necessary condition for reversibility.

Consider a configuration $c=c_{0,0}=(D,x,y)$, $\gamma$ an $\SLE$ from $x$ to $y$ and $\hat\gamma$ an $\SLE$ from $y$ to $x$. Denote $c_{s,t}=(D\setminus (K_s\cup\hat K_t),\gamma_s,\hat\gamma_t)$. Let $U,\hat U$ be disjoint neighbourhoods of $x,y$ respectively; $\tau,\hat\tau$ denote first exits of $U,\hat U$ by $\gamma,\hat\gamma$ respectively.
Assume that $\gamma,\hat\gamma$ can be coupled so that one is the reversal of the other. Then an application of the Markov property for $\gamma,\hat\gamma$ shows that the distribution of $\gamma^\tau$ conditional on $\hat\gamma^{\hat\tau}$ is (stopped) $\SLE$ in $c_{0,\hat\tau}=(D\setminus\hat K_{\hat\tau},x,\hat\gamma_{\hat\tau})$. Symmetrically, the conditional distribution of $\hat\gamma^{\hat\tau}$ given $\gamma^\tau$ is $\SLE$ in $c_{\tau,0}=(D\setminus K_\tau,\gamma_\tau,y)$. By integration, this gives the identity of measures:
\begin{equation}\label{Comm}
\int \hat f(\hat\gamma^{\hat\tau})(\int f(\gamma^\tau) d\mu_{c_{0,\hat\tau}}^U(\gamma^\tau))d\hat\mu^{\hat U}_{c}(\hat\gamma^{\hat\tau})=
\int f(\gamma^\tau) (\int \hat f(\hat\gamma^{\hat\tau})d\hat\mu_{c_{\tau,0}}^{\hat U}(\hat\gamma^{\hat\tau}))d\mu^{U}_{c}(\gamma^{\tau})
\end{equation}
for arbitrary positive Borel functions $f,\hat f$. This is the {\em local commutation} condition studied in \cite{Dub6}. Disintegrating and inserting densities (that exist from absolute continuity properties) yields the condition:
\begin{equation}\label{Comm2}
\left(\frac{d\mu_{c_{0,\hat\tau}}^U}{d\mu_c^U}\right)(\gamma^\tau)=\left(\frac{d\hat\mu_{c_{\tau,0}}^{\hat U}}{d\hat\mu_c^{\hat U}}\right)(\hat\gamma^{\hat\tau})
\end{equation}
almost everywhere in $\gamma^\tau,\hat\gamma^{\hat\tau}$. This is an identity between two (continuous) functions of the paths $\gamma,\hat\gamma$. From the above results on absolute continuity of $\SLE$ (Proposition \ref{Pdens}), we see that both sides are indeed equal, and their common value is the explicit quantity:
\begin{equation}\label{RN}
\ell_D(\gamma^\tau,\hat\gamma^{\hat\tau})=\left(\frac{H(c_{\tau,\hat\tau})H(c)}{H(c_{\tau,0})H(c_{0,\hat\tau})}\right)^\alpha\exp(-\lambda m(D;K_\tau,\hat K_{\hat\tau}))
\end{equation}
which is manifestly symmetric in $\gamma^\tau,\hat\gamma^{\hat\tau}$. (We use $\ell$ for {\em likelihood ratio}, somewhat abusively).

Using the expression of $\ell$ as Radon-Nikodym derivatives of probability measures, we see that:
\begin{align}\label{Coupl}
\int \ell_D(\gamma^\tau,\hat\gamma^{\hat\tau})d\hat\mu_c^{\hat U}(\hat\gamma^{\hat\tau})=\int d\hat\mu_{c_{\tau,0}}(\hat\gamma^{\hat\tau})=1&&\forall \gamma^\tau\nonumber\\
\int \ell_D(\gamma^\tau,\hat\gamma^{\hat\tau})d\mu_c^{U}(\gamma^{\tau})=\int d\mu_{c_{0,\hat\tau}}(\gamma^{\tau})=1&&\forall \hat\gamma^{\hat\tau}
\end{align}

This also applies to any pair of stopping times $\sigma,\hat\sigma$ dominated by $\tau,\hat\tau$ (ie $\sigma$ is a stopping time for $\gamma$ such that $\sigma\leq\tau$ a.s.).

Such local commutation identities (without relying on explicit densities) are proved in greater generality in \cite{Dub6} under an infinitesimal commutation condition, which is easily checked in the present case.

Let us also point out a (deterministic) tower property, dictated by compatibility with the SLE Markov property. Let $(D,x,y)$ be a domain, $(K_s)$ a Loewner chain growing at $x$ (with trace $\gamma_s$), and $(\hat K_t)$ a Loewner chain growing at $x$ (with trace $\hat\gamma_t$). Let $c_{s,t}=(D\setminus (K_s\cup \hat K_t),\gamma_s,\hat\gamma_t)$. Let us denote, for $0\leq s_1\leq s_2$, $0\leq t_1\leq t_2$:
$$\ell_{s_1,t_1}^{s_2,t_2}=\ell_{c_{s_1,t_1}}(\gamma^{s_2},\hat\gamma^{t_2})$$
Then, for $0\leq s_1\leq s_2\leq s_3$, $0\leq t_1\leq t_2\leq t_3$:
\begin{equation}\label{tower}
\ell_{s_1,t_1}^{s_2,t_2}\ell_{s_2,t_1}^{s_3,t_2}=\ell_{s_1,t_1}^{s_3,t_2},{\textrm\ \ \ \ }\ell_{s_1,t_1}^{s_2,t_2}\ell_{s_1,t_2}^{s_2,t_3}=\ell_{s_1,t_1}^{s_2,t_3}.
\end{equation}
For fixed $\gamma$, the first relation has to hold a.e. in $\gamma$ to ensure compatibility of \eqref{Coupl} with the Markov property of $\gamma$; the second relation corresponds to the Markov property of $\hat\gamma$. Alternatively, this can be checked directly from the explicit expression \eqref{RN}, by telescopic cancellations and the restriction property of the loop measure $\mu^{loop}$ (\cite{LW}).

\subsection{The general case}

We now move to the general case (in simply connected domains) of local commutation, following Theorem 7.1 of \cite{Dub6}, which we rephrase in the present context.

A configuration consists of a simply connected domain $D$ with marked points: $c=(D,z_0,z_1,\dots,z_n,z_{n+1})$; the marked points are distinct and in some prescribed order on the boundary. The question is to classify pairs of $\SLE$ (with the $\SLE$ Markov property relatively to these configurations), one growing at $z_0$ (hulls $(K_s)$, trace $\gamma$), one at $z_{n+1}$ (hulls $(\hat K_t)$, trace $\hat\gamma$), that satisfy local commutation \eqref{Comm}, \eqref{Comm2}. As before, we denote $c_{s,t}=(D_{s,t}=D\setminus (K_s\cup\hat K_t),\gamma_s,z_1,\dots,z_n,\hat\gamma_t)$ when this is still a configuration (ie before swallowing of any marked point). 

We take as reference measures a chordal $\SLE_\kappa$ from $z_0$ to $z$, and a chordal $\SLE_{\hat\kappa}$ from $z_{n+1}$ to $z$, where $z$ is another marked boundary point, used solely for normalization. Then:

\begin{Thm}[\cite{Dub6}]\label{class}
Local commutation is satisfied iff: $\hat\kappa\in\{\kappa,16/\kappa\}$ and there exists a conformally invariant function $\psi$ on the configuration space and exponents $\nu_{ij}$ such that $\sum_{j=1}^{n+1}\nu_{0,j}=\alpha_\kappa$, $\sum_{i=0}^n\nu_{i,n+1}=\alpha_{\hat\kappa}$, and if
$$Z(c)=\psi(c)\prod_{0\leq i<j\leq n+1} H_D(z_i,z_j)^{\nu_{ij}}$$
and
\begin{align*}
M_s&=H_c(z_0,z)^{-\alpha_\kappa}Z(c_{s,0})\\
\hat M_t&=H_c(z_{n+1},z)^{-\alpha_{\hat\kappa}}Z(c_{0,t})
\end{align*}
then $(M_s)$ is a local martingale for $\SLE_\kappa(D,z_0,z)$, $(\hat M_t)$ is a local martingale for $\SLE_{\hat\kappa}(D,z_{n+1},z)$; these martingales are the densities of the commuting $\SLE$'s w.r.t. the reference chordal $\SLE$ measures, up to swallowing of a marked point. 
\end{Thm}

In the theorem, notice that $M_s=H_c(z_0,z)^{-\alpha_\kappa}Z(c_{s,0})$ is defined via a choice of local coordinates at $z,z_1,\dots,z_{n+1}$ , but not at $z_0$, where the evolution occurs (the choice of local coordinates is arbitrary but fixed under evolution).

A good example of the situation is the following: $\hat\kappa=\kappa=6$, with four marked points $(z_0,z_1,z_2,z_3$), $\psi$ the probability that there is a percolation crossing from $(xy)$ to $(z_1z_2)$.

The absolute continuity properties of these $\SLE$'s reduce to that of chordal $\SLE$ as in Lemma \ref{RNrho}: let $Z(c)=\psi(c)\prod_{i<j}H_D(z_i,z_j)^{\nu_{ij}}$ (this depends on local coordinates at the $z_i$'s; appropriate ratios will not depend on coherent choices). Then we have
\begin{equation*}
\ell_D(\gamma^\tau,\hat\gamma^{\hat\tau})=\left(\frac{Z(c_{\tau,\hat\tau})Z(c_{0,0})}{Z(c_{\tau,0})Z(c_{0,\hat\tau})}\right)\exp(-\lambda m(D;K_\tau,\hat K_{\hat\tau}))
\end{equation*}
where, if $\mu_c$, $\hat\mu_c$ denote the two commuting $\SLE$ measures, $\tau$, $\hat\tau$ stopping times preceding swallowing of any marked point,
$$\ell_D(\gamma^\tau,\hat\gamma^{\hat\tau})=\left(\frac{d\mu_{c_{0,\hat\tau}}^\tau}{d\mu_c^\tau}\right)(\gamma^\tau)=\left(\frac{d\hat\mu_{c_{\tau,0}}^{\hat\tau}}{d\hat\mu_c^{\hat\tau}}\right)(\hat\gamma^{\hat\tau}).$$
Notice that when $\hat\kappa\in\{\kappa,16/\kappa\}$, $\lambda_\kappa=\lambda_{\hat\kappa}$. From compatibility with the $\SLE$ Markov property, or directly from the expression \eqref{RNgen}, we still have the tower property (see \eqref{tower}).

The definition of $Z$ involves $1$-jets of local coordinates at marked points; one could imagine more complicate dependences, say on $k$-jets at marked points; this is ruled out by the theorem. The situation in other (non simply connected) topologies is quite involved, though an important part of the analysis carries through (see \cite{Dub6}). For $n$ $\SLE$'s in a simply connected domain, the local commutation relation for the system of $n$ $\SLE$'s reduces to $\frac{n(n-1)}2$ pairwise commutation conditions.

An easy way to generate systems of commuting $\SLE$'s is to look for partition functions in the simple form $Z(c)=\prod_{i<j}H_D(z_i,z_j)^{\nu_{ij}}$. For an admissible choice of exponents $\nu_{ij}$, this corresponds to an $\SLE_\kappa(\underline\rho,\rho)$ in $(\H,z_0,\dots,z_{n+1})$ with $\rho_i=-2\kappa\nu_{0,i}$, $i=1\dots n$; $\rho=\rho_{n+1}=-2\kappa\nu_{0,n+1}$. Similarly, the other $\SLE$ is an $\SLE_{\hat\kappa}(\hat\rho,\underline{\hat\rho})$ with $\hat\rho=\hat\rho_0=-2\hat\kappa\nu_{0,n+1}$, and $\hat\rho_i=-2\hat\kappa\nu_{i,n+1}$.

This situation is studied in \cite{Dub6}, Section 3.2. The following systems are found to solve the local commutation condition: 

\begin{enumerate}
\item $\kappa=\hat\kappa$; $\rho=\hat\rho=\kappa-6$; $\rho_i=\hat\rho_i=0$ (two chordal $\SLE$'s aiming at each other, the reversibility setup)
\item $\kappa=\hat\kappa$; $\rho=\hat\rho=2$; $\rho_i=\hat\rho_i$, $\rho+\sum_i\rho_i=\kappa-6$ ($n-1$ arbitrary parameters)
\item $\kappa\hat\kappa=16$, $\rho=-\kappa/2$, $\hat\rho=-\frac{\hat\kappa} 2$, $\hat\rho_i=-(\hat\kappa/4)\rho_i=-(4/\kappa)\rho_i$, $\rho+\sum_i\rho_i=\kappa-6$ ($n-1$ arbitrary parameters)
\end{enumerate}

In these cases, we have the following expressions for the common partition function (see Lemma \ref{rhodens}):
\begin{align*}
Z(c)&=H(z_0,z_{n+1})^{\frac{6-\kappa}{2\kappa}}\\
Z(c)&=H(z_0,z_{n+1})^{-\frac 1\kappa}\prod_i H(z_0,z_i)^{-\frac{\rho_i}{2\kappa}}H(z_{n+1},z_i)^{-\frac{\rho_i}{2\kappa}}\prod_{i<j}H(z_i,z_j)^{-\frac{\rho_i\rho_j}{4\kappa}}\\
Z(c)&=H(z_0,z_{n+1})^{-\frac 14}\prod_i H(z_0,z_i)^{-\frac{\rho_i}{2\kappa}}H(z_{n+1},z_i)^{\frac{\rho_i}{8}}\prod_{i<j}H(z_i,z_j)^{-\frac{\rho_i\rho_j}{4\kappa}}
\end{align*}

In the third case ($\kappa\hat\kappa=16$), notice that $-\frac{\rho_i}{2\kappa}=\frac{\hat\rho_i}8$, $-\frac{\rho_i\rho_j}{4\kappa}=-\frac{\hat\rho_i\hat\rho_j}{4\hat\kappa}$.

Another explicit situation is when four points are marked, so that there is a single cross-ratio. Consider a configuration $c=(D,x,y,z_1,z_2)$, the marked points in some prescribed order. Let $\nu$ be a parameter and $\beta$ a solution of the quadratic equation $\frac\kappa 2\beta(\beta-1)+2\beta=2\nu$. Define a partition function
$$Z(c)=H_D(x,y)^{\frac {6-\kappa}{2\kappa}}H_D(z_1,z_2)^\nu\psi(u)$$
where $u$ is the cross-ratio $u=\frac{(z_1-x)(z_2-y)}{(y-x)(z_2-z_1)}$ (in the upper half-plane) and 
$$\psi(u)=(u(1-u))^\beta\vphantom{F}_2F_1(2\beta,2\beta+\frac 8\kappa-1;2\beta+\frac 4\kappa;u)$$
where $\vphantom{F}_2F_1$ designates a solution of the hypergeometric equation with parameters $2\beta,\dots$ (this equation is invariant under $u\leftrightarrow 1-u$). If the solution is chosen so that it is positive on the configuration space, then a computation shows that $Z$ satisfies the condition of the theorem and drives two locally commuting $\SLE$'s starting from $x,y$. This covers for instance the following situations: a chordal $\SLE_\kappa$ from $x$ to $y$ conditioned not to intersect the interval $[z_1,z_2]$, $4<\kappa<8$; a chordal $\SLE_{8/3}$ from $x$ to $y$ conditioned not to intersect a restriction measure (with exponent $\nu$) from $z_1$ to $z_2$; and the marginal of a system of two $\SLE$ strands $x\leftrightarrow y$, $z_1\leftrightarrow z_2$ (\cite{Dub7}, Section 4.1; this corresponds to $\nu=\alpha_\kappa$).

\section{Coupling}

Let $c=(D,z_0,z_1,\dots,z_n,z_{n+1})$ be a configuration, where $D$ is a simply connected, bounded domain with $n+2$ distinct marked points on the boundary in some prescribed order. We consider a system of two $\SLE$'s satisfying local commutation, one originating at $z_0$, the other at $z_{n+1}$. These two $\SLE$'s have the $\SLE$ Markov property for domains with $n+2$ marked points. The first one is absolutely continuous (up to a disconnection event) w.r.t. $\SLE_\kappa(D,z_0,z_n)$; the Loewner chain is $(K_s)$, the trace $\gamma$, the measure $\mu_c$. The second one is absolutely continuous (up to a disconnection event) w.r.t. $\SLE_{\hat\kappa}(D,z_{n+1},z_n)$; the Loewner chain is $(\hat K_t)$, the trace $\hat \gamma$, the measure $\hat\mu_c$. We also denote $c_{s,t}=(D\setminus (K_s\cup\hat K_t),\gamma_s,z_1,\dots,z_n,\hat\gamma_t)$.

From Theorem
\ref{class} and the following discussion, we know that $\hat\kappa\in\{\kappa,16/\kappa\}$ and there is a positive conformally invariant function $\psi$ on configurations, weights $\nu_{ij}$, such that for $\tau,\hat\tau$ a pair of stopping times (before disconnection events)
\begin{align}\label{RNgen}
\ell_c(\gamma^\tau,\hat\gamma^{\hat\tau})
&=\left(\frac{Z(c_{\tau,\hat\tau})Z(c_{0,0})}{Z(c_{\tau,0})Z(c_{0,\hat\tau})}\right)\exp(-\lambda m(D;K_\tau,\hat K_{\hat\tau}))\nonumber\\
&=\left(\frac{d\mu_{c_{0,\hat\tau}}^\tau}{d\mu_c^\tau}\right)(\gamma^\tau)
=\left(\frac{d\hat\mu_{c_{\tau,0}}^{\hat\tau}}{d\hat\mu_c^{\hat\tau}}\right)(\hat\gamma^{\hat\tau})
\end{align}
where $Z(c)=\psi(c)\prod_{i<j}H_D(z_i,z_j)^{\nu_{ij}}$. 

Let $\hat\gamma^{\hat\tau}$ be a fixed stopped path. Let $\tau$ be a stopping time for $\gamma$ conditional on $\hat\gamma^{\hat\tau}$, such that a.s. $\gamma^\tau$ is at distance at least $\eta>0$ of $\hat\gamma^{\hat\tau}$ and other marked points. Then we deduce immediately from \eqref{RNgen} that:
\begin{align}\label{Couplgen}
\int \ell_c(\gamma^\tau,\hat\gamma^{\hat\tau})d\mu_c^{\tau}(\gamma^{\tau})=\int d\mu_{c_{0,\hat\tau}}^\tau(\gamma^{\tau})=1&&\forall \hat\gamma^{\hat\tau}
\end{align}
The symmetric statement holds for the same reason.

Define:
$$\ell_{s_1,t_1}^{s_2,t_2}=\ell_{c_{s_1,t_1}}(\gamma^{s_2},\hat\gamma^{t_2})=\left(\frac{Z(c_{s_2,t_2})Z(c_{s_1,t_1})}{Z(c_{s_2,t_1})Z(c_{s_1,t_2})}\right)\exp(-\lambda m(D\setminus (K_{s_1}\cup\hat K_{t_1});K_{s_2},\hat K_{t_2})).$$
We recall the (deterministic) tower property: 
for $0\leq s_1\leq s_2\leq s_3$, $0\leq t_1\leq t_2\leq t_3$ such that $c_{s_3,t_3}$ is a configuration
\begin{equation}\label{towergen}
\ell_{s_1,t_1}^{s_2,t_2}\ell_{s_2,t_1}^{s_3,t_2}=\ell_{s_1,t_1}^{s_3,t_2},{\textrm\ \ \ \ }\ell_{s_1,t_1}^{s_2,t_2}\ell_{s_1,t_2}^{s_2,t_3}=\ell_{s_1,t_1}^{s_2,t_3}.
\end{equation}
The goal of this section is to use these properties to extend a local coupling of $\mu_c,\hat\mu_c$, that exists from local commutation, to a maximal coupling ``up to disconnection". The idea, introduced in \cite{DZrevers} in the case of reversibility, is to consider a sequence of couplings absolutely continuous w.r.t. the independent coupling $\mu_c\otimes\hat\mu_c$ that converges to a maximal coupling. The construction of the coupling presented here somewhat differs from that in \cite{DZrevers}. The existence of a such a maximal coupling relies solely on 
\eqref{Couplgen}, \eqref{towergen}, and the $\SLE$ Markov property.

\subsection{Local coupling}

We briefly discuss here the interpretation of local commutation \eqref{RNgen} in terms of couplings.

A coupling of $\mu_c,\hat\mu_c$ is a measure on pairs of paths $(\gamma,\hat\gamma)$ (or chains $(K,\hat K)$) such that the first marginal is $\mu_c$, the second marginal is $\hat\mu_c$. The trivial (independent) coupling is $\mu_c\otimes\hat\mu_c$.

For simplicity, consider $U,\hat U$ disjoint neighbourhoods of $x,y$ respectively in $D$ at positive distance of other marked points. Let $\tau,\hat\tau$ be the first exit of $U,\hat U$ by $\gamma,\hat\gamma$.
Then we can consider the measure on pairs of stopped paths $(\gamma^\tau,\hat\gamma^{\hat\tau})$ given by:
$$\ell_c(\gamma^\tau,\hat\gamma^{\hat\tau})d\mu_c^\tau(\gamma^\tau)d\hat\mu_c^{\hat\tau}({\hat\gamma}^{\hat\tau})$$
From \eqref{Couplgen}, we know that this is a coupling of the stopped measures $\mu_c^\tau,\hat\mu_c^{\hat\tau}$. One can extend it to a coupling of $\mu_c$, $\hat\mu_c$ as follows: after $\tau$, $\gamma$ is continued as an $\SLE$ in $c_{\tau,0}$ (that is, following $\mu_{c_{\tau,0}}$) independent of the rest conditionally on $c_{\tau,0}$; $\hat\gamma$ is continued in a similar way. This describes a coupling of $\mu_c$, $\hat\mu_c$, using the (strong) Markov property. It is clear that this procedure describes the measure:
$$\ell_c(\gamma^\tau,\hat\gamma^{\hat\tau})d\mu_c(\gamma)d\hat\mu_c({\hat\gamma})$$
on pairs of paths $(\gamma,\hat\gamma)$. This coupling is local in the sense that the interaction is restricted to the time set $[0,\tau]\times [0,\hat\tau]$. In the next subsection, we will extend the interaction to the (random) time set $\{(s,t):\ K_s\cap\hat K_t=\varnothing\}$.

\subsection{Maximal coupling}

The goal here is to construct explicit couplings of $\mu_c,\hat\mu_c$ parameterized by a small parameter $\eta>0$ such that any subsequential limit as $\eta\searrow 0$ is a maximal coupling.

We assume $\eta\ll\dist(z_0,z_{n+1})\leq \diam(D)<\infty$. (The domain $D$ is embedded in the plane and distances are measured in the ambient plane).

Define by induction the stopping times: $\tau_0=0$; 
$$\tau_{i+1}=\inf\{t\geq 0: \gamma_t\notin (K_{\tau_i})^\eta\}$$
where $A^\eta=\{x\in D: \dist(x,A)<\eta\}$. Then $\tau_i$ is a function of the path $\gamma$ (taking a modification where traces exist and are continuous). The sequence $\tau_i$ is strictly increasing until it reaches $\infty$. 

If $\tau_n<\infty$, then for all $i<j\leq n$, $\dist(\gamma_{\tau_i},\gamma_{\tau_j})\geq\eta$, since $\gamma_{\tau_j}$ is outside of $(K_{\tau_{j-1}})^\eta$. Thus $n\pi(\frac\eta 2)^2\leq area(D+B(0,\eta))$. This gives a fixed $N=\lfloor 4area(D+B(0,\eta))/\pi\rfloor$ such that $\tau_N=\infty$.

Similarly, define $\hat\tau=0$, and by induction:
$$\hat\tau_{j+1}=\inf\{t\geq 0: \hat\gamma_t\notin (\hat K_{\hat\tau_j})^\eta\}.$$
We now introduce a dependence on the other path. Let $\partial$ be the smallest connected boundary arc of $D$ containing all marked points except $z_0$; symmetrically, $\hat\partial$ be the smallest connected boundary arc of $D$ containing all marked points except $z_{n+1}$. These two arcs overlap in general.

We define a (random) set $G\subset \N^2$ of good pairs of indices as follows:
$$\{(i,j)\in G\}=\{\dist(K_{\tau_i},\hat K_{\hat\tau_j})>3\eta\}\cap\{\dist(K_{\tau_i},\partial)>2\eta\}\cap\{\dist(\hat K_{\hat\tau_j},\hat\partial)>2\eta\}.$$
(Here, $\dist(K,\hat K)=\inf_{x\in K,\hat x\in \hat K}(\dist(x,\hat x))$). This is a separation condition. If $i'\leq i$, $j'\leq j$, $(i,j)\in G$, then a fortiori $(i',j')\in G$. We take $\eta$ small enough so that $(0,0)\in G$. 

For conciseness, denote $$\ell_{ij}^{i'j'}=\ell_{\tau_i,\hat\tau_j}^{\tau_{i+1},\hat\tau_{j+1}}=\ell_{c_{\tau_i,\hat\tau_j}}(\gamma^{\tau_{i'}},\hat\gamma^{\hat\tau_{j'}})$$
for $i\leq i'$, $j\leq j'$. 
If $\tau_i=\tau_{i'}$, set $\ell_{ij}^{i'j'}=1$ (this is the case if $i=i'$ or $\tau_i=\tau_{i'}=\infty$).

Let $(i,j)\in G$. Then $K_{\tau_i}$ is at distance at least $3\eta$ from $\hat K_{\hat\tau_j}$ and at least $2\eta$ from $\partial$. Thus $K_{\tau_{i+1}}\subset \overline{(K_{\tau_i})^\eta}$ is at distance at least $2\eta$ from $\hat K_{\hat\tau_j}$, at least $\eta$ from $\hat K_{\hat\tau_{j+1}}$, and at least $\eta$ from $\partial$. Symmetric statements hold for $\hat K_{\hat\tau_j}$. Thus $\ell_{i,j}^{i+1,j+1}$ is well defined and by Lemma \ref{Lbound} uniformly bounded by some $C=C(D,\eta)$.

Define
\begin{align*}
I(j)&=\inf\{i\in\N: (i,j)\notin G\}\\
J(i)&=\inf\{j\in\N: (i,j)\notin G\}.
\end{align*}
Plainly, $I$ and $J$ are nonincreasing.
If ${\mc F}$ (resp. $\hat{\mc F}$) is the filtration generated by $\gamma$ (resp. $\hat\gamma$), $\tau_i$ is an ${\mc F}$-stopping time. Also, the event $\{(i,j)\in G\}$ is $({\mc F}_{\tau_i}\vee\hat{\mc F}_{\hat\tau_j})$-measurable. It follows that $\tau_{I(j)}$ is a stopping time in the enlarged filtration $({\mc F}_t\vee\hat{\mc F}_{\hat\tau_j})_t$. Symmetric statements hold for $\hat\tau_j$, $\hat\tau_{J(i)}$.

Consider the measure $\Theta^\eta_c$ on pairs $(\gamma,\hat\gamma)$:
$$d\Theta^\eta_c(\gamma,\hat\gamma)=\left [\prod_{(i,j)\in G}\ell_{i,j}^{i+1,j+1}\right]d\mu_c(\gamma)d\hat\mu_{\hat c}(\hat\gamma)$$

The density is a function on pairs of paths $(\gamma,\hat\gamma)$. 
From the tower property (\ref{tower}), one can rewrite this density as:
$$L=\prod_{(i,j)\in G}\ell_{i,j}^{i+1,j+1}=\prod_{0\leq i\leq N}\ell_{i,0}^{i+1,J(i)}=\prod_{0\leq j\leq N}\ell_{0,j}^{I(j),j+1}$$
showing in particular that $L$ is bounded by $C^N$. The last two expressions of $L$ will behave well in combination with the Markov property of $\gamma$, $\hat\gamma$ respectively.

\begin{Lem}
The measure $\Theta_c^\eta=L(\mu_c\otimes\hat\mu_c)$ is a coupling of $\mu_c,\hat\mu_c$.
\end{Lem}

\begin{proof}
The statement can be rephrased as
\begin{align*}
\forall \hat\gamma,\E(L)=1&&
\forall \gamma,\hat\E(L)=1
\end{align*} 
where $\E$, $\hat\E$ refer to integration w.r.t. $d\mu_c(\gamma)$, $d\hat\mu_{\hat c}(\hat\gamma)$ respectively (in other terms $\E(.)=\E\otimes\hat\E(.|\hat{\mc F}_\infty)$). The situation is completely symmetric, so we shall consider only the distribution of the second marginal; that is, we have to check that given any $\hat\gamma$, $\E(L)=1$.

For this, the relevant expression of the density is: $L=\prod_{0\leq i\leq N}\ell_{i,0}^{i+1,J(i)}$. Notice that $\hat\gamma^{\hat\tau_{J(i)}}$ is $({\mc F}_{\tau_i}\vee\hat{\mc F}_\infty)$-measurable.
This implies that a term $\ell_{i,0}^{i+1,J(i)}$ is $({\mc F}_{\tau_n}\vee\hat{\mc F}_\infty)$-measurable for any $i<n$. Thus for fixed $n$:
$$\E\left(\prod_{i=0}^n\ell_{i,0}^{i+1,J(i)}|{\mc F}_{\tau_n}\right)=\left[\prod_{i=0}^{n-1}\ell_{i,0}^{i+1,J(i)}\right]\E(\ell_{n,0}^{n+1,J(n)}|{\mc F}_{\tau_n})$$
and $\E(\ell_{n,0}^{n+1,J(n)}|{\mc F}_{\tau_n})=1$ by (\ref{Couplgen}); the point is that given $({\mc F}_{\tau_n}\vee\hat{\mc F}_\infty)$, $\gamma^{\tau_n}$ and $\hat\gamma^{\hat\tau_{J(n)}}$ are fixed. %
This is saying that 
$$M_n=\prod_{i=0}^{n-1}\ell_{i,0}^{i+1,J(i)}$$
is a (discrete time, bounded) martingale in the filtration $({\mc F}_{\tau_n}\vee\hat{\mc F}_\infty)_{n\geq 0}$. In particular, $\E(L)=\E(M_N)=M_0=1$.
\end{proof}

We can actually get a more precise statement. Let $\tau$ be an arbitrary ${\mc F}$-stopping time; let $n=\inf\{i\in\N:\tau_i\geq\tau\}$, a random integer. We assume that $\dist(\gamma^\tau,\partial)\geq 3\eta$ a.s., so that $(n,0)\in G$. Then $\tau_n$ is a stopping time approximating $\tau$ (as $\eta\searrow 0$). Consider the joint distribution of $(\gamma^{\tau_n},\hat\gamma^{\hat\tau_{J(n)}})$ under $\Theta^\eta$. Then $\gamma^{\tau_n}$ has distribution $\mu_c^{\tau_n}$. Moreover:
\begin{align*}
\E\otimes\hat\E(L|{\mc F}_{\tau_n}\vee\hat{\mc F}_{\hat\tau_{J(n)}})&=\left[\prod_{\substack{i<n\\j<J(n)}}\ell_{i,j}^{i+1,j+1}\right]\E\otimes\hat\E\left(\prod_{\substack{i\geq n\\ j<J(i)}}\ell_{i,j}^{i+1,j+1}\prod_{\substack{j\geq J(n)\\i<I(j)}}\ell_{i,j}^{i+1,j+1}|{\mc F}_{\tau_n}\vee\hat{\mc F}_{\hat\tau_{J(n)}}\right)\\
&=\left[\ell_{0,0}^{n,J(n)}\right]\E\otimes\hat\E\left(\prod_{i\geq n}\ell_{i,0}^{i+1,J(i)}\prod_{j\geq J(n)}\ell_{0,j}^{I(j),j+1}|{\mc F}_{\tau_n}\vee\hat{\mc F}_{\hat\tau_{J(n)}}\right)
\end{align*}
The term $\prod_{i\geq n}\ell_{i,0}^{i+1,J(i)}$ involves $\gamma$ after $\tau_n$ and $\hat\gamma$ before $\hat\tau_{J(n)}$ (if $i\geq n$, $J(i)\leq J(n)$); the other term, $\prod_{j\geq J(n)}\ell_{0,j}^{I(j),j+1}$, involves $\hat\gamma$ after $\hat\tau_{J(n)}$ and $\gamma$ before $\tau_n$. Hence these two terms are independent conditionally on $({\mc F}_{\tau_n}\vee\hat{\mc F}_{\hat\tau_{J(n)}})$ (under the independent measure $\mu_c\otimes\hat\mu_c$). 
Moreover
$$\E(\prod_{i\geq n}\ell_{i,0}^{i+1,J(i)}|{\mc F}_{\tau_n})=\E(\frac{M_N}{M_n}|{\mc F}_{\tau_n})=1$$
where $(M_.)$ is the discrete-time, bounded martingale considered in the previous lemma, since $n$ is a stopping time for its discrete filtration. Similarly, $\hat\E(\prod_{j\geq J(n)}\ell_{0,j}^{I(j),j+1}|\hat{\mc F}_{\hat\tau_{J(n)}})=1$, and consequently:
\begin{align*}
\E\otimes\hat\E\left(L|{\mc F}_{\tau_n}\vee\hat{\mc F}_{\hat\tau_{J(n)}}\right)&=(\ell_{0,0}^{n,J(n)})\E\otimes\hat\E\left(\prod_{i\geq n}\ell_{i,0}^{i+1,J(i)}|{\mc F}_{\tau_n}\vee\hat{\mc F}_{\hat\tau_{J(n)}}\right)\E\otimes\hat\E\left(\prod_{j\geq J(n)}\ell_{0,j}^{I(j),j+1}|{\mc F}_{\tau_n}\vee\hat{\mc F}_{\hat\tau_{J(n)}}\right)\\
&=(\ell_{0,0}^{n,J(n)})\hat\E\left[\E\left[\prod_{i\geq n}\ell_{i,0}^{i+1,J(i)}|{\mc F}_{\tau_n}\right]|\hat{\mc F}_{\hat\tau_{J(n)}}\right]
\E\left[\hat\E\left[\prod_{j\geq J(n)}\ell_{0,j}^{I(j),j+1}|\hat{\mc F}_{\hat\tau_{J(n)}}\right]|{\mc F}_{\tau_n}\right]\\
&=\ell_{0,0}^{n,J(n)}=
\frac{d\hat\mu_{c_{\tau_n,0}}^{\hat\tau_{J(n)}}}{d\hat\mu_{c}^{\hat\tau_{J(n)}}}(\hat\gamma^{\hat\tau_{J(n)}}).
\end{align*}
This proves that under $\Theta^\eta$, the conditional distribution of $\hat\gamma^{\hat\tau_{J(n)}}$ given $\gamma^{\tau_n}$ is 
$\hat\mu_{c_{\tau_n,0}}^{\hat\tau_{J(n)}}$, where $\hat\tau_{J(n)}$ is a stopping time conditionally on $\gamma^{\tau_n}$.

To phrase the following theorem, it is convenient to introduce:
\begin{align*}
\sigma&=\sup\{t\geq 0: K_t\cap \partial=\varnothing\} \\
\hat\sigma&=\sup\{t\geq 0: \hat K_t\cap \hat\partial=\varnothing\}
\end{align*}
The measures $\mu_c$, $\hat\mu_c$ are defined on paths $(\gamma,\hat\gamma)$ up to $\sigma,\hat\sigma$.

\begin{Thm}\label{maxcoupl}
Let $\mu_c,\hat\mu_c$ be $\SLE$ measures in a configuration $(D,z_0,\dots,z_{n+1})$ satisfying local commutation.
Then there exists a coupling $\Theta$ of $\mu_c$, $\hat\mu_c$ which is maximal in the following sense:

for any ${\mc F}$-stopping time $\tau$, $\tau\leq\sigma$, let $\hat\tau$ be the $({\mc F}_\tau\vee\hat{\mc F}_t)_t$-stopping time:
$$\hat\tau=\sup\{t\geq 0: \hat K_t\cap (K_\tau\cup\hat\partial)=\varnothing\}.$$
Then under $\Theta$, the pair $(\gamma^\tau,\hat\gamma^{\hat\tau})$ has the following distribution: $\gamma^\tau$ is distributed according to $\mu_c^\tau$, and conditionally on $\gamma^\tau$, $\hat\gamma^{\hat\tau}$ is distributed according to $\hat\mu^{\hat\tau}_{c_{\tau,0}}$.
The symmetric statement holds.
\end{Thm}

\begin{proof}
For any $\eta>0$, the two marginal distributions of $\Theta^\eta$ are fixed Radon measures (in the topology of Carath\'eodory convergence of Loewner chains). Thus the family $(\Theta^\eta)_{\eta>0}$ is tight, and Prokhorov's theorem ensures existence of subsequential limits. 
Let $(\eta_k)_k$ be a sequence $\eta_k\searrow 0$ along which $\Theta^{\eta_k}$ has a weak limit $\Theta$. Then $\Theta$ is a coupling of $\mu_c$ and $\hat\mu_{c}$.
We can consider a probability space with sample $((\gamma^1,\hat\gamma^1),\dots,(\gamma^k,\hat\gamma^k),\dots)$ such that the distribution of $(\gamma^k,\hat\gamma^k)$ is $\Theta^{\eta_k}$ and $(\gamma^k,\hat\gamma^k)\rightarrow (\gamma,\hat\gamma)$ a.s., where the distribution of $(\gamma,\hat\gamma)$ is $\Theta$.

Let $\tau$ be an ${\mc F}$-stopping time; we approximate $\tau$ in a convenient way. Firstly, $\tau$ can be approximated by $\tau'$ taking values in some discrete countable sequence $(t_i)_{i\geq 0}$ (eg dyadic times). Hence there are Borel sets $B_i$ such that $\ind_{B_i}$ is a Borel function of $K^{t_i}$ and $\tau'=\inf\{t_i: K_{t_i}\in B_i\}$. Replace the Borel set $B_i$ by a larger open set $U_i$ such that the measure of $U_i\setminus B_i$ is very small. Then  $\tau''=\inf\{t_i: K_{t_i}\in U_i\}$ is a stopping time equal to $\tau'$ with probability arbitrarily close to 1. Finally, let $\tau'''=\tau''\wedge\sup\{t:\dist(\gamma_t,\partial)\geq\eps'\}$ for some fixed $\eps'>0$.
Let us assume for now that $\tau$ is of type $\tau'''$. This gives a common stopping rule for all the chains $K^k_.$: stop the first time that $K^k_{t_i}$ is in $U_i$ or at distance $\eps$ of $\partial$. We denote $\tau^k$ this stopping time for the chain $K^k_.$. In particular, $\tau^k\rightarrow\tau$ a.s. (using that the $U_i$'s are open).

For $\eta>0$, $\tau^k_n$ is an approximation of $\tau^k$ as above:
$\tau^k_n=\inf\{\tau_i: \tau_i\geq\tau^k\}$; then 
$$K_{\tau^k_{n-1}}^k\subset K^k_{\tau^k}\subset K^k_{\tau^k_n}\subset (K^k_{\tau^k_{n-1}})^{\eta_k}.$$
It is easy to see that $\tau^k_n\rightarrow\tau$, $\tau^k_{n-1}\rightarrow\tau$, $K^k_{\tau^k_n}\rightarrow K_\tau$ and $\gamma^k_{\tau^k_n}\rightarrow\gamma_\tau$ (since $\gamma_\tau=\cap_{s>0}\overline{K_{\tau+s}\setminus K_\tau}$) as $k\rightarrow\infty$.
We have seen that the conditional distribution of $\hat\gamma_k^{\hat\tau_{J(n)}}$ is $\hat\mu_{c_{\tau_n,0}}^{\hat\tau_{J(n)}}$. Notice that $\hat\tau_{J(n)}$ occurs after first entrance in $(K^k_{\tau^k_n})^{3\eta_k}$ and before entrance in $(K^k_{\tau^k_n})^{\eta_k}$. 

For fixed $\eps>0$, $K^k_{\tau^k_n}\subset (K_\tau)^\eps$ for $k$ large enough. The configuration $\hat c^k_{\tau^k_n,0}$ converges in the Carath\'eodory topology to $c_{\tau,0}$ (with also convergence of $\gamma^k_{\tau^k_n}$ to $\gamma_\tau$); this implies weak convergence of the conditional distribution of $\hat\gamma_k$ stopped when entering $(K_\tau)^\eps$ to the corresponding stopped $\SLE$ in $c_{\tau,0}$. This gives the correct conditional distribution of $\hat\gamma$ stopped when entering $(K_\tau)^\eps$, conditional on $\gamma^\tau$. One concludes by taking $\eps\searrow 0$.

This proves the result for a dense set of stopping times of type $\tau'''$ as above (this will be enough to draw geometric consequences). A general stopping time $\tau\leq\sigma$ is the limit of a sequence of stopping times $\tau'''_m$; for each $m$, the conditional distribution of $\hat\gamma$ stopped upon entering $(K_{\tau'''_m})^\eps$ is correct. One concludes by taking $m\rightarrow\infty$ and then $\eps\searrow 0$.
\end{proof}

There are some obvious extensions of this result. One involves radial $\SLE$'s (not necessarily aiming at the same bulk point). Another involves systems of $n$ (pairwise) commuting $\SLE$'s. Let us discuss this case briefly.

Consider a configuration $c=(D,z_1,\dots,z_n,z_{n+1},\dots,z_{n+m})$, with $n$ $\SLE$'s starting at $z_1,\dots,z_n$, driven by the same partition function $Z$. One can reason as above (sampling the $\SLE$'s at discrete times $\tau^1_{i_1},\dots,\tau^n_{i_n}$). In a maximal coupling, one can stop the first $\SLE$ at a stopping time $\tau^1$, the second at $\tau^2$ (first time it ceases to be defined or meets $K^1_{\tau^1}$), \dots , the $n$-th at $\tau^n$ (first time it ceases to be defined or meets $\cup_{i=1}^{n-1}K^i_{\tau^i}$) and get the appropriate joint distribution. This works for any permutation of indices.

Let us describe the local coupling in this case: let $U_1,\dots,U_n$ be disjoint neighbourhoods of $z_1,\dots,z_n$, $\mu^1_.,\dots,\mu^n_.$ the commuting $\SLE$ measures, $Z$ their common partition function. 
Let $K_i$ be the hull of the stopped $i$-th $\SLE$; 
$c_{\eps_1\dots\eps_n}$, $\eps_i\in\{0,1\}$, is the configuration where the $i$-th $\SLE$ has grown (until stopped) if $\eps_i=1$. Consider the density
$$L=\frac{d\mu^n_{c_{1\dots10}}}{d\mu^n_{c_{0\dots 0}}}\cdot\frac{d\mu^{n-1}_{c_{1\dots100}}}{d\mu^{n-1}_{c_{0\dots 0}}}\cdots\frac{d\mu^2_{c_{10\dots 0}}}{d\mu^2_{c_{0\dots 0}}}=\frac{Z(c_{1\dots 1})Z(c_{0\dots 0})^{n-1}}{Z(c_{10\dots 0})\dots Z(c_{0\dots 01})}\exp\left(-\lambda\sum_{j=2}^n m(D,\cup_{i=1}^{j-1}K_i,K_j)\right)$$
Then it is clear from the first expression that the first marginal of $L(\mu_c^1\otimes\cdots\otimes \mu_c^n)$ is $\mu_c^1$ (integrating out $K_n$, then $K_{n-1}$, \dots); the second expression shows that the construction is symmetric (for a discussion of the loop measure contribution, see Section 3.4 of \cite{Dub7}).

\section{Geometric consequences}

We have proved (Theorem \ref{maxcoupl}) existence of maximal couplings under a local commutation assumption. On the other hand, the systems of $\SLE$'s satisfying this assumption are classified (Theorem \ref{class}). So we can now apply the existence of maximal couplings to appropriate systems of commuting $\SLE$'s to extract information on the geometry of $\SLE$ curves.

\subsection{Reversibility}

Reversibility for $\kappa\in (0,4]$ is proved in \cite{DZrevers}. We review the result for the reader's convenience.

\begin{Thm}
If $\kappa\leq 4$, $\SLE$ is reversible; any maximal coupling $\Theta$ of chordal $\SLE_\kappa$ in $(D,x,y)$ with $\SLE_\kappa$ in $(D,y,x)$ is the coupling of $\SLE$ with its reverse trace.
\end{Thm}
\begin{proof}
Let $(D,x,y)$ be a configuration, $\mu_c$ the chordal $\SLE$ measure from $x$ to $y$, $\hat\mu_c$ the chordal $\SLE$ measure from $y$ to $x$. They satisfy local commutation, hence there exists a maximal coupling $\Theta$.

Take a countable dense sequence of ${\mc F}$-stopping times $(\tau^m)$ (e.g., capacity of the hull reaches a rational number); denote simply by $\tau$ an element in this sequence. Then in the maximal coupling $\Theta$, the conditional distribution of $\hat\gamma^{\hat\tau}$ is $\SLE$ in $(D\setminus K_\tau,y,\gamma_\tau)$ stopped upon hitting $K_\tau$. For $\kappa\leq 4$, the $\SLE$ trace intersects the boundary only at its endpoints. Hence $\hat\gamma_{\hat\tau}=\gamma_\tau$. This proves that under $\Theta$, the intersection $\gamma\cap\hat\gamma$ is a.s. dense in $\gamma$. Since both $\gamma,\hat\gamma$ are a.s. closed, $\gamma\subset\hat\gamma$; since $\hat\gamma$ is simple, removing a point disconnects it, but $\gamma$ is connected. Hence the occupied sets of $\gamma,\hat\gamma$ are equal. Again, as the paths are simple, the occupied set determines the parameterized trace. Hence in any maximal coupling $\Theta$, $\hat\gamma=\gamma^r$  (the reverse trace) a.s.; this determines the coupling uniquely.
\end{proof}

Besides local commutation, the argument uses only qualitative properties of the paths. So we can phrase at no additional cost:

\begin{Cor}
Let $\kappa\leq 4$, $\mu_c$, $\hat\mu_c$ a system of commuting SLE in the configuration $c=(D,z_0,z_1,\dots,z_n,z_{n+1})$. Assume that $\mu_c$, $\hat\mu_c$ are supported on simple paths that meet the boundary of $D$ only at $z_0,z_{n+1}$. Then $\gamma^r$ and $\hat\gamma$ are identical in distribution. 
\end{Cor}

A simple example of the situation is as follows: let $(z_1,\dots,z_4)$ be four marked points on the arc $(z_0z_5)$. Then we can consider chordal $\SLE$ from $z_1$ to $z_5$ weighted by any, say, bounded above and below function of the cross-ratio of $(z_1,\dots,z_4)$ in $D\setminus\gamma$. This plainly preserves both local commutation and reversibility. 

Another setup where the corollary applies is the following: $c=(D,z_0,z_1,\dots,z_{n+1})$ with points in counterclockwise order. Let $\rho_1,\dots,\rho_n$ be such that $\rho_1+\cdots+\rho_i\geq 0$ for $1\leq i<n$ and $\rho_1+\cdots+\rho_n=0$. Then the traces of $\SLE_4(\underline\rho,-2)$ starting from $z_0$ and $\SLE_4(-2,-\underline\rho)$ starting from $z_{n+1}$ are the reverse of each other in distribution. This describes the scaling limit of the zero level line of a discrete free field (\cite{SS_freefield}) with piecewise constant boundary conditions (with jump at $z_i$ proportional to $\rho_i$). A version with marked points on both sides of $z_0$ also holds.

One also obtains reversibility identities for the pairs of commuting $\SLE$'s (aiming at each other) with four marked points described at the end of Section 3.2. By degenerating two points into one, this describes the reversal of $\SLE_\kappa(\rho)$, $\kappa\leq 4$, $\rho\geq\frac\kappa 2-2$. For instance, if $\kappa=8/3$, one can represent an $\SLE_{8/3}(\rho)$ in $(\H,0,1,\infty)$ as the limit of a chordal $\SLE_{8/3}$ in $(\H,0,\infty)$ conditioned not to intersect a restriction measure with exponent $\nu=\nu(\rho)$ from $1$ to $z\gg 1$ (\cite{W2}; reversibility in this case follows from \cite{LSW3}). For general $\kappa$, it is unclear whether there is a simple probabilistic interpretation, but one still gets an exact (if unwieldy in general) description of the reversal. 

\begin{Cor}
Let $\kappa\leq 4$, $\rho\geq\frac\kappa 2-2$, $(D,x,y)$ a configuration. Then $\SLE_\kappa(\rho)$ in $(D,x,y,x^+)$ and in $(D,y,x,y^-)$ have the same occupied set in distribution, where $x,x^+,y^-,y$ are in this order on the boundary.
\end{Cor}
\begin{proof}
We sketch the argument. The result follows from reversibility in the regular situation with four marked points described at the end of Section 3.2. Indeed, if $x,z_1,z_2,y$ are in this order on the boundary, one has a pair of commuting $\SLE$'s starting at $x,y$ with common partition function: $Z(c)=H_D(x,y)^{\frac {6-\kappa}{2\kappa}}H_D(z_1,z_2)^\nu\psi(u)$
where $u$ is the cross-ratio $u=\frac{(z_1-x)(z_2-y)}{(y-x)(z_2-z_1)}$ (in the upper half-plane), $\frac\kappa 2\beta(\beta-1)+2\beta=2\nu$ and 
$$\psi(u)=(u(1-u))^\beta\vphantom{F}_2F_1(2\beta,2\beta+\frac 8\kappa-1;2\beta+\frac 4\kappa;u).$$
When $\kappa\leq 4$, $\rho=\kappa\beta\geq\frac\kappa 2-2$, the processes do not hit $[z_1,z_2]$, by comparison arguments. Thus the local commutation extends to a maximal coupling, and in this coupling the occupied sets coincide. 

The first $\SLE$ is the martingale transform of chordal $\SLE_\kappa$ in $(D,x,y)$ by the martingale (in upper half-plane coordinates):
$$t\longmapsto \left(\frac{g'_t(z_1)g'_t(z_2)}{(g_t(z_1)-g_t(z_2))^2}\right)^\nu\psi(u_t)$$
where $u_t$ is the cross-ratio at time $t$. Take $z_2=y-\eps$. Then the leading term of the expansion of the martingale as $\eps\searrow 0$ is:
$$t\longmapsto \left(\frac{g'_t(z_1)g'_t(y)}{(g_t(z_1)-g_t(y))^2}\right)^\nu\left(\frac{(g_t(z_1)-X_t)g'_t(y)}{(g_t(y)-X_t)(g_t(y)-g_t(z_1))}\right)^\beta$$
so that this limiting process is identified from Lemma \ref{rhodens} as $\SLE_\kappa(\rho)$ in $(D,x,y,z_1)$, $\rho=\kappa\beta$. A symmetric result holds for the other $\SLE$.
\end{proof}

The same arguments can be used to establish reversibility of systems of multiple $\SLE$'s considered in \cite{Dub7} (this also follows from the symmetry of the density of the system w.r.t. independent chordal $\SLE$'s when the pairing of endpoints is fixed).

\subsection{Duality}

The question of $\SLE$ duality is to describe boundaries of $\SLE_\kappa$, $\kappa>4$, in terms of $\SLE_{\hat\kappa}$, $\hat\kappa=16/\kappa$. 

There are various parametric situations we can consider. Let us start with the simplest setting: a configuration $c=D(x,z_1,y,z_2)$ has four marked points $x,y,z_1,z_2$ on the boundary. We consider two $\SLE$'s (inducing the measures $\mu_c$, $\hat\mu_c$, with traces $\gamma$, $\hat\gamma$), see Table 1 ($[\kappa]$ represents an $\SLE_\kappa$ ``seed", the other entries are the $\rho$ parameters). 
\begin{table}[htdp]
\caption{}
\begin{center}
\begin{tabular}{|c|c|c|c|}
\hline
$x$&$z_1$&$y$&$z_2$\\
\hline
$[\kappa]$&$\rho_1$&$-\frac\kappa 2$&$\rho_2$\\
\hline
$-\frac{\hat\kappa}2$&$\hat\rho_1$&$[\hat\kappa]$&$\hat\rho_2$\\
\hline
\end{tabular}
\end{center}
\label{default}
\end{table}%

The additional conditions for local commutation are $\rho_1+\rho_2=\frac 32(\kappa-4)$, $\hat\rho_i=-\frac 4\kappa\rho_i$, consequently $\hat\rho_1+\hat\rho_2=\frac 32(\hat\kappa-4)$. This leaves one free parameter, say $\rho_1=\rho$. We need to put conditions on $\rho$ so that paths have a correct geometry. Take $\rho\in [\frac{\kappa-4}2,\kappa-4]$, a nonempty interval when $\kappa>4$. Consequently, $\rho_2\in [\frac{\kappa-4}2,\kappa-4]$, $\hat\rho_1,\hat\rho_2\in [\frac{\hat\kappa-4}2,\hat\kappa-4]$. Then the first $\SLE$ will first intersect $[z_1,z_2]$ at $y$ (see Lemma \ref{rhohit}). To restrict even more the situation, take $\hat\rho_1=\hat\kappa-4$ (or symmetrically $\hat\rho_2=\hat\kappa-4$). Then the second $\SLE$ cannot hit $(z_2,x)$, nor $(z_1,z_2)$ except at $y$; and it hits $(x,z_1)$ since $\hat\rho_1=\hat\kappa-4<\frac{\hat\kappa}2-2$.

\begin{Prop}\label{Pdual}
In a maximal coupling $\Theta$ of $\mu_c$, $\hat\mu_c$, the range of $\hat\gamma$ is contained in that of $\gamma$. If $\rho_1=\kappa-4$, $\hat\gamma$ is the right boundary of $K$; if $\rho_2=\kappa-4$, $\hat\gamma$ is the left boundary of $K$.
\end{Prop}
\begin{proof}
As before, take $\hat\tau$ a stopping time for the second $\SLE$. The first $\SLE$ in $c_{0,\hat\tau}$ is defined until it exits at $\hat\gamma_{\hat\tau}$. More precisely, it is defined up to a time where it accumulates at $\hat\gamma_{\hat\tau}$ and at no other point of the boundary arc $[z_1,z_2]$ of $c_{0,\hat\tau}$ (Lemma \ref{rhohit}). But $\gamma$ is continuous away from $[z_1,z_2]$ in $c_0$; so if $\hat\tau$ is positive, $\gamma$ stopped when exiting $c_{0,\hat\tau}$ has a limit, which is $\hat\gamma_{\hat\tau}$. Hence $\hat\gamma_{\hat\tau}$ is on $\gamma$.
Taking countably many stopping times, this shows that $\hat\gamma$ is included in (the range of) $\gamma$. Moreover, ordering is preserved: $\hat\gamma_{\hat\tau}=\gamma_t$ for some $t\leq\tau$, and for any $t<\hat\tau$, $\hat\gamma_t\notin K_\tau$.

Set $\rho_1=\kappa-4$. Then the range of $\gamma$ is partitioned in points on $\hat\gamma$, to its left, or to its right. Take a stopping time $\tau$. Then $\hat\gamma$ first hits $K_\tau$ on the arc $[\gamma_\tau,z_1]$. If $\gamma_\tau$ was to the right of $\hat\gamma$, then $\hat\gamma$ would have to circle $\gamma_\tau$ and reenter in $K_\tau$, which would violate the ordering condition. Hence a generic point $\gamma_\tau$ is on $\hat\gamma$ or to its left. This implies that $\hat\gamma$ is contained in the right boundary of the range of $\gamma$, which is a simple path. Since $\hat\gamma$ starts at $y$ (where $\gamma$ ends) and ends on $(x,z_1)$, this shows that $\hat\gamma$ is the right boundary of the range of $\gamma$. 
\end{proof}

\begin{Rem}
The situation where $\rho$ varies in $[\frac{\kappa-4}2,\kappa-4]$ is of some independent interest and seems related to pivotal points questions.
\end{Rem}

We consider now versions where the non simple $\SLE$ is actually chordal $\SLE_\kappa$, at the expense of some complication for the dual simple $\SLE_{\hat\kappa}$.

\begin{proof}[Proof of Theorem \ref{dual4}.]
 
 Assume that $\kappa\in (4,8)$. Consider chordal $\SLE_\kappa$, say in $(\H,0,\infty)$. The point 1 is swallowed at time $\tau_1$; $D=\gamma_{\tau_1}$ is on $(1,\infty)$ with distribution given by:
$$\P(D\in(1,z))=F(z)=c\int_1^z u^{-\frac 4\kappa}(u-1)^{\frac 8\kappa-2}du$$
where $c=B(1-4/\kappa,8/\kappa-1)^{-1}$. In other words, $D^{-1}$ has a $\Beta(1-4/\kappa,8/\kappa-1)$ distribution. 
The function $F$ is such that $t\mapsto F((g_t(z)-W_t)/(g_t(1)-W_t))$ is a martingale. Let us disintegrate the $\SLE$ measure w.r.t. $D$ (see \cite{Dub4} for related questions). It is easy to see that up to $\tau_1$, the $\SLE$ conditional on $D\in dz$ is the martingale transform of chordal $\SLE$ by:
$$t\mapsto \partial_z F\left(\frac{g_t(z)-W_t}{g_t(1)-W_t}\right)=c\frac{g'_t(z)}{g_t(1)-W_t}(g_t(z)-W_t)^{-\frac 4\kappa}(g_t(z)-g_t(1))^{\frac 8\kappa-2}(g_t(1)-W_t)^{2-\frac 4\kappa}$$
and this is readily identified with $\SLE_\kappa(\kappa-4,-4)$ in $(\H,0,\infty,1,z)$ (Lemma \ref{rhodens}). To get a regular situation, we split the point $z$ into two points $y$ and $z_2$, while setting $x=0$, $z_1=1$, $z_3=\infty$. Consider the system of commuting $\SLE$'s given by Table 2.
\begin{table}[htdp]
\caption{$4<\kappa<8$}
\begin{center}
\begin{tabular}{|c|c|c|c|c|}
\hline
$x$&$z_1$&$y$&$z_2$&$z_3$\\
\hline
$[\kappa]$&$\kappa-4$&$-\frac\kappa 2$&$\frac\kappa 2-4$&$2$\\
\hline
$-\frac{\hat\kappa}2$&$\hat\kappa-4$&$[\hat\kappa]$&$\hat\kappa-2$&$-\frac{\hat\kappa}2$\\
\hline
\end{tabular}
\end{center}
\label{default}
\end{table}%

The first $\SLE$ hits $[z_1,z_3]$ at $y$, while the second $\SLE$ will not hit $[z_2,z_3]$ and exits $[x,z_1]$ somewhere in $(x,z_1)$ (Lemma \ref{rhohit}). Arguing as in Proposition \ref{Pdual}, this shows that $\hat\gamma$ is the right boundary of $K$. Finally, one takes $z_2\searrow y$, so that the first $\SLE_\kappa$ becomes chordal $\SLE_\kappa$ conditional on $D=y$. This yields Theorem \ref{dual4}.
\end{proof}

When $\kappa\geq 8$, the trace is a.s. space filling, and we have to proceed differently to isolate a boundary arc.

\begin{proof}[Proof of Theorem \ref{dual8}.]
Consider now the case of a chordal $\SLE_\kappa$ in $(\H,0,\infty)$, $\kappa\geq 8$ (thus $\hat\kappa\leq 2$). Then $\gamma_{\tau_1}=1$ a.s. There is a leftmost point $G$ on $(\infty,0)$ visited by the trace before $\tau_1$. We are interested in the boundary of $K_{\tau_G}$, a simple curve from $G$ to some point in $(0,1)$. Then the distribution of $G$ is given by:
$$\P(G\in (z,0))=c\int_z^0 (-u)^{-4/\kappa}(1-u)^{\frac 8\kappa-2}du$$
where $c=B(1-4/\kappa,1-4/\kappa)$
In other words, $G$ is such that $G/(G-1)$ has a $\Beta(1-4/\kappa,1-4/\kappa)$ distribution (generalised arcsine distribution).
The disintegrated $\SLE$ measure w.r.t. $G$ is again $\SLE_\kappa(-4,\kappa-4)$ in $(\H,0,\infty,G,1)$, up to hitting $G$. To get a regular situation, we need to split the point $G$ into three points $z_1,y,z_2$; we also set $x=0$, $z_3=1$, $z_4=\infty$. Consider the system of two commuting $\SLE$'s in $(\H,y,z_1,x,z_2,z_3)$ given by Table 3.
\begin{table}[htdp]
\caption{$\kappa\geq 8$}
\begin{center}
\begin{tabular}{|c|c|c|c|c|c|}
\hline
$z_1$&$y$&$z_2$&$x$&$z_3$&$z_4$\\
\hline
$-2$&$-\frac\kappa 2$&$\frac\kappa 2-2$&$[\kappa]$&$\kappa-4$&$2$\\
\hline
$\frac{\hat\kappa}2$&$[\hat\kappa]$&$\frac{\hat\kappa}2-2$&$-\frac{\hat\kappa}2$&$\hat\kappa-4$&$-\frac{\hat\kappa}2$\\
\hline
\end{tabular}
\end{center}
\end{table}%

The first $\SLE$ exits at $y$, the second one exits in $(x,z_3)$ (Lemma \ref{rhohit}). As in Proposition \ref{Pdual}, this shows that one can couple the two $\SLE$'s such that the second one is the boundary arc of the first one between $y$ and a point of $(x,z_3)$. Taking $z_1\nearrow y$, $z_2\searrow y$ gives Theorem \ref{dual8}.
\end{proof}

At the expense of some complications, one can consider more symmetric situations. Let $(D,x,y,z,z',y',x')$ be a configurations (points are in that order). There is a system of four commuting $\SLE$'s attached to this configuration (where $a+b=2$), see Table 4.
\begin{table}[htdp]
\caption{}
\begin{center}
\begin{tabular}{|c|c|c|c|c|c|}
\hline
$x$&$y$&$z$&$z'$&$y'$&$x'$\\
\hline
$[\kappa]$&$a(\kappa-4)$&$-\frac\kappa 2$&$-\frac\kappa 2$&$b(\kappa-4)$&$2$\\
\hline
$2$&$a(\kappa-4)$&$-\frac\kappa 2$&$-\frac\kappa 2$&$b(\kappa-4)$&$[\kappa]$\\
\hline
$-\frac{\hat\kappa}2$&$a(\hat\kappa-4)$&$[\hat\kappa]$&$2$&$b(\hat\kappa-4)$&$-\frac{\hat\kappa}2$\\
\hline
$-\frac{\hat\kappa}2$&$a(\hat\kappa-4)$&$2$&$[\hat\kappa]$&$b(\hat\kappa-4)$&$-\frac{\hat\kappa}2$\\
\hline
\end{tabular}
\end{center}
\end{table}%

\section{Some technical results}

\subsection{Absolute continuity for variants of $\SLE$}

In this subsection, we phrase similar absolute continuity results for different versions of $\SLE$. In the context of duality, it is useful to consider $\SLE$-type measures in a parametric family $\SLE_\kappa(\underline\rho)$ (\cite{LSW3,Dub4}), as acknowledged in \cite{Dub4}.

An $\SLE_\kappa(\underline\rho)$, $\underline\rho=\rho_1,\dots,\rho_n$, in the configuration $(\H,x,\infty,z_1,\dots,z_n)$ is an $\SLE$ the driving process of which satisfies the SDE:
$$dW_t=\sqrt\kappa dB_t+\sum_{i=1}^n\frac{\rho_i}{W_t-g_t(z_i)}dt$$
and $W_0=x$, up to swallowing of a $z_i$. See Lemma 3.2 of \cite{Dub6} for homographic change of coordinates. In particular, if $\sum_i\rho_i=\kappa-6$, the point at infinity is used for normalization only.

The following lemma is a change of measure result 
(see also \cite{W2}). 

\begin{Lem}\label{rhodens}
Consider an $\SLE_\kappa$ starting from $x$ in $\H$, 
$\underline\rho=\rho_1,\dots,\rho_n$; let $Z^i_t=g_t(z_i)-W_t$. Then:
$$M_t=\prod_i g'_t(z_i)^{\alpha_i}|Z^i_t|^{\beta_i}\prod_{i<j}|Z^j_t-Z^i_t|^{\eta_{ij}}$$
is a local martingale if $2\alpha_i=\frac\kappa 2\beta_i(\beta_i-1)+2\beta_i$, $2\eta_{ij}=\kappa\beta_i\beta_j$.
Before the swallowing of any marked point, $M_t/M_0$ is the density of an $\SLE_\kappa(\underline\rho)$ starting from $(x,z_1,\dots z_n)$ w.r.t. $\SLE_\kappa$, where $\underline\rho=\kappa\beta_1,\dots,\kappa\beta_n$. 
\end{Lem}
\begin{proof}
This is a standard computation relying on:
\begin{align*}
dZ^i_t=\frac 2{Z^i_t}dt-\sqrt\kappa dB_t&&
\frac{dg'_t(z_i)}{g'_t(z_i)}=-\frac 2{(Z^i_t)^2}dt&&
\frac{d(Z^j_t-Z^i_t)}{(Z^j_t-Z^i_t)}=-\frac{2}{Z^i_tZ^j_t}dt
\end{align*}

so that:
\begin{align*}
\frac{dM_t}{M_t}&=\sum_i\frac{\beta_i}{Z^i_t}\left(\frac 2{Z^i_t}dt-\sqrt\kappa dB_t\right)
+\frac{\kappa}2\frac{\beta_i(\beta_i-1)}{(Z^i_t)^2}dt-\frac{2\alpha_i}{(Z^i_t)^2}dt+\sum_{i<j}(\kappa\beta_i\beta_j-2\eta_{ij})\frac{dt}{Z^i_tZ^j_t}
\end{align*}
The statement on the density follows from the Girsanov theorem (e.g. \cite{RY}), observing that:
$$\frac{d\langle{M_t,W_t}\rangle}{M_t}=-\sum_i\frac{\kappa\beta_i}{Z^i_t}dt$$
that is, the drift term of an $\SLE_\kappa(\underline\rho)$ with $\rho_i=\kappa\beta_i$. More precisely, under the original measure, $W=\sqrt\kappa B$, $B$ a standard Brownian motion. Under the transformed measure (via the local martingale $M_t$ stopped away from swallowing a $z_i$), $\hat W=W-\langle W,M\rangle /M$ is a (local) martingale with the same quadratic variation as $W$; ie, from L\'evy's theorem, a Brownian motion $\sqrt\kappa \hat B$. Hence $W=\sqrt\kappa\hat B+\langle W,M\rangle /M$.
\end{proof}

Let $c=(D,z_0,\dots,z_n)$ be a configuration. As in Theorem \ref{class}, we consider a variant of $\SLE$ of the following type:
let $\psi$ be a positive, continuous, conformally invariant function on the configuration space and exponents $\nu_{ij}$ such that if
$$Z(c)=\psi(c)\prod_{0\leq i<j\leq n+1} H_D(z_i,z_j)^{\nu_{ij}}$$
then $\sum_{j=1}^{n+1}\nu_{0,j}=\alpha_\kappa$ and
$$M_s=H_D(z_0,z)^{-\alpha_\kappa}Z(c_{s,0})$$
is a local martingale for the reference measure $\SLE_\kappa(D,z_0,z)$, $z$ an auxiliary marked point on the boundary. For short, let us denote $\SLE_\kappa(Z)$ obtained by Girsanov transform of the reference $\SLE_\kappa(D,z_0,z)$ by $M$ (up to a disconnection event).

For example, from Lemma \ref{rhodens}, it is easy to see that $\SLE_\kappa(\underline\rho)$ in $c=(D,z_0,z_1,\dots,z_n)$, $\rho_1+\cdots+\rho_n=\kappa-6$, is $\SLE_\kappa(Z)$ with:
$$Z(c)=\prod_{i=1}^n H_D(z_0,z_i)^{-\frac{\rho_i}{2\kappa}}\prod_{1\leq i<j\leq n}H_D(z_i,z_j)^{-\frac{\rho_i\rho_j}{4\kappa}}.$$

The following is the analogue of Proposition \ref{Pdens}. 

\begin{Lem}\label{RNrho}
Let $c=(D,z_0,z_1,\dots,z_n)$ be a configuration consisting of a simply connected domain $D$ with $n+1$ marked points on the boundary; $c'=(D',z_0,z'_1,\dots,z'_n)$ is another configuration agreeing with $D$ in a neighbourhood $U$ of $z_0$; $U$ is at positive distance of marked points other than $z_0$. Let $\overline\mu_c^U$ denote the distribution of an $\SLE_\kappa(Z)$ in $c$, stopped upon exiting $U$. Then:
$$\frac{d\overline\mu^U_{c'}}{d\overline\mu^U_{c}}(\gamma)=\left(\frac{Z(c'_\tau)Z(c)}{Z(c_\tau)Z(c')}\right)\exp(-\lambda m(D;K_\tau,D\setminus D')+\lambda m(D';K_\tau,D'\setminus D))$$ 
where $c_\tau=(D\setminus K_\tau,\gamma_\tau,z_1,\dots,z_n)$, similarly for $c'_\tau$.
\end{Lem}
\begin{proof}

One can proceed as follows: let $\mu^U_c$ denotes chordal $\SLE_\kappa$ in the configuration $c$ (aiming at an auxiliary point $z$), stopped upon exiting $U$. Then trivially:
$$\frac{d\overline\mu^U_{c'}}{d\overline\mu^U_{c}}=\frac{d\overline\mu^U_{c'}}{d\mu^U_{c'}}\cdot\frac{d\mu^U_{c'}}{d\mu^U_{c}}\cdot\frac{d\mu^U_{c}}{d\overline\mu^U_{c}}$$
The middle term is studied in Proposition \ref{Pdens}, while the outer terms are, from the definition of $\SLE_\kappa(Z)$:
 $$\frac{d\mu^U_{c}}{d\overline\mu^U_{c}}=\frac{M_\tau}{M_0}=\frac{Z(c_{\tau})}{Z(c)}\cdot\frac{H_D(z_0,z)^\alpha}{H_{D_\tau}(\gamma_\tau,z)^\alpha}$$
 where $\tau$ is the first exit of $U$, and similarly for the other term. Under the assumptions above, $M^\tau$ is uniformly bounded (see Lemma \ref{Lbound}).
\end{proof}

Recall that $Z(c)$ depends on a choice of local coordinates at the marked points as a 1-form; but the ratio $\frac{Z(c'_\tau)Z(c)}{Z(c_\tau)Z(c')}$ does not depend on the choices.

\subsection{A bound on densities}

We give an upper bound on Radon-Nikodym derivatives that appear in the coupling argument. This is a rough estimate that is sufficient for our purposes.

A configuration $c=(D,x,y,z_1,\dots,z_n)$ consists in a bounded simply connected Jordan domain $D$, with distinct marked points on its  boundary; $\partial$ (resp. $\hat\partial$) is the smallest connected boundary arc containing all marked points except $x$ (resp. $y$); $K$ (resp. $\hat K$) is a chain growing at $x$ (resp. $y$) generated by the continuous trace $\gamma$ (resp. $\hat\gamma$). We denote $c_{s,t}=(D\setminus (K_s\cup \hat K_t),\gamma_s,z_1,\dots z_n,\hat\gamma_t)$; also $Z(c)=\psi(c)\prod_{i<j}H_D(z_i,z_j)^{\nu_{ij}}$, $\psi$ a positive, continuous, conformally invariant function. For $0\leq s'\leq s, 0\leq t'\leq t$, define:
$$\ell_{s',t'}^{s,t}=\left(\frac{Z(c_{s,t})Z(c_{s',t'})}{Z(c_{s',t})Z(c_{s,t'})}\right)\exp(-\lambda m(D\setminus (K_{s'}\cup\hat K_{t'});K_{s},\hat K_{t}))$$
\begin{Lem}\label{Lbound}
For any $\eta>0$ small enough, there exists $C=C(D,\eta)>0$ such that for all chains $K,\hat K$, $0\leq s'\leq s, 0\leq t'\leq t$ with $\dist(K_{s},\hat K_{t})\geq \eta$, $\dist(K_{s},\partial)\geq\eta$, $\dist(\hat K_{t},\hat\partial)\geq\eta$,
$$C^{-1}<\ell_{s',t'}^{s,t}<C$$
\end{Lem}
\begin{proof}
From the identity: $\ell_{s',t'}^{s,t}=\ell_{0,0}^{s',t'}\ell_{0,0}^{s,t}(\ell_{0,0}^{s',t}\ell_{0,0}^{s,t'})^{-1}$, it is enough to prove the bound for $s'=t'=0$. From e.g. Corollary 2.8 in \cite{Pommerenke}, it is enough to prove it in any reference Jordan domain, say the upper semidisk, with all marked points on the segment $(-1,1)$. Also without loss of generality, one can assume there is at least one marked point $z_1$.

In the bounded domain $D$, the total mass of loops of diameter at least $\eta$ in the loop measure $\mu^{loop}$ is finite; this gives uniform bounds above and below for the factor $\exp(-\lambda m(\dots))$.

Consider the set $S$ of quadruplets $(K,x',\hat K,y')$ where $K,\hat K$ are compact subsets of $\overline D$, $K,\hat K$ connected, with $x',y'$ on their respective boundaries, $\dist(K,K')\geq\eta$, $\dist(K,\partial)\geq\eta$, $\dist(\hat K,\hat\partial)\geq\eta$, and $x'$ (resp. $y'$) corresponds to a single prime end on $D\setminus (K\cup\hat K)$. (This last condition is always satisfied ``at the tip"). The set $S$ is compact (for Hausdorff convergence of compact subsets of $\overline D$). To such a quadruplet are associated four configurations: $c_{0,0}=c$, $c_{1,0}=(D\setminus K,x',y,z_1,\dots,z_n)$, $c_{0,1}=(D\setminus\hat K,x,y',\dots)$, $c_{1,1}=(D\setminus (K\cup\hat K,x',y',\dots)$. Then the ratio $\frac{Z(c_{11})Z(c_{00})}{Z(c_{10})Z(c_{01})}$ defines a positive function on the compact set $S$. It is enough to prove that this function is continuous. 

Let $(K_n,x'_n,\hat K_n,y'_n)_n$ converge to $(K,x',\hat K,y')$. By Schwarz reflection across $[-1,1]$ and the Carath\'eodory convergence theorem (\cite{Pommerenke}, Theorem 1.8), the conformal equivalence $\phi^n_{11}$ between $D^n_{11}=D\setminus(K^n\cup\hat K_n)$ and $D_{11}=D\setminus (K\cup\hat K)$, extended by reflection and normalized by $\phi^n_{11}(z_1)=z_1$, $\phi^n_{11}(z_1)>0$, converges locally uniformly away from the unit circle and $(K\cup \overline K\cup\hat K\cup{\overline{\hat K}})$ (here $\overline K$ is the conjugate of $K$). The same holds for $D_{10}^n=D\setminus K_n$ and $D_{01}^n=D\setminus\hat K_n$.

Fix small semidisks $D(z_i,\eta/2)$ around the marked points $z_i$'s; the choice of $D$ as the upper semidisk gives a choice of local coordinates at the $z_i$'s. Then the $H_{D^n_{..}}(z_i,z_j)$ are numbers; they can be decomposed into: excursion harmonic measure in the semidisks $D(z_i,\eta/2)$, $D(z_j,\eta/2)$, and the Green function at points on $C(z_i,\eta/2)$, $C(z_j,\eta/2)$. The excursion harmonic measures are fixed and the Green function converge due to the conformal invariance of the Green function and uniform convergence of the $\phi_{..}^n$ near the $z_i$'s. This proves continuity of the $H_{D^n_{..}}(z_i,z_j)$.

The treatment of ratios of type $H_{D^n_{1,1}}(x'_n,z_i)/H_{D^n_{1,1}}(x'_n,z_j)$ is similar: take a crosscut $\delta$ at positive distance of $K$, separating it from $\hat K$ and the marked points. Then the Poisson excursion kernel can be decomposed w.r.t. the first crossing (by a Brownian motion starting near $x'_n$) of $\delta$ and the last crossing of $C(z_i,\eta/2)$. It is easy to see that the excursion harmonic measure on $\delta$ converges. This gives continuity of ratios of type $H_{D^n_{1,1}}(x'_n,z_i)/H_{D^n_{1,1}}(x'_n,z_j)$. Assuming without loss of generality that there are at least 2 marked points $z_1,z_2$, the term $H(x,y)^.$ can be eliminated from the partition function.

The only remaining thing to check is the convergence of the cross ratios. This is immediate for those not involving $x,y$, from the convergence of the $\phi^n_{..}$ as above. This can be done also for those involving $x,y$; though for our purposes it is enough to prove that all cross-ratios (between marked points) are uniformly bounded. By comparison arguments, it is enough to prove it for cross-ratios involving $x^-,x^+$, $y^-,y^+$ (instead of $x,y,x',y'$) , where these new points
are on $[-1,1]$ and are such that the interval $(x^-,x^+)$ (resp. $(y^-,y^+)$) contains $K_n\cap [-1,1]$ (resp. $\hat K_n\cap [-1,1]$) for $n$ large enough, and no other marked point. This then reduces to the previous situation.

\end{proof}

\subsection{First exit of $\SLE_\kappa(\underline\rho)$}

We need to establish some simple qualitative properties of $\SLE_\kappa(\underline\rho)$ in, say, a reference configuration $c=(\H,0,\infty,z_1,\dots,z_n)$. In particular, we are interested in the position of the trace the first time a marked point is swallowed.


Assume that $n=2$, $0<z_1<z_2<\infty$. Then the $\SLE$ is well defined up to swallowing of $z_1$ at time $\tau_1=\tau_{z_1}$. There are several possibilities: $\tau_1=\infty$; $\gamma_{\tau_1}=z_1$; $\gamma_{\tau_1}\in (z_1,z_2)$; $\gamma_{\tau_1}=z_2$; $\gamma_{\tau_1}\in (z_2,\infty)$; or $\gamma_{\tau_1}$ does not exist. (This last case is unlikely to ever happen, though delicate to rule out in general).

More precisely, let $Y_t=\frac{g_t(z_1)-W_t}{g_t(z_2)-W_t}$ and $ds=\frac{dt}{(g_t(z_2)-W_t)^2}=-\frac 12 d\log g_t'(z_2)$. Then:
$$dY_s=(1-Y_s)\left[\sqrt\kappa dB_s+(\frac{\rho_1+2}{Y_s}+\rho_2+2-\kappa)dt\right]$$
where $B$ is a standard Brownian motion. This is a diffusion on $[0,1]$. Notice that $g'_t(z_2)$ is positive before $\tau_2$, and goes to zero at $t\nearrow \tau_2$. A scale function of this diffusion is $F$:
$$F(y)=\int_{1/2}^y u^{-\frac 2\kappa(2+\rho_1)}(1-u)^{\frac 2\kappa(4+\rho_1+\rho_2-\kappa)}du.$$
It blows up at 0 if $\rho_1\geq\frac\kappa 2-2$. This means that $Y$ does not reach 0 in finite time, so that $\tau_1=\tau_2$ a.s. (possibly infinite). If $\rho_1<\frac\kappa 2-2$, $\rho_1+\rho_2\leq\frac\kappa 2-4$, the scale function blows up at 1, not 0, meaning that $\tau_1<\tau_2$ a.s.

Assume that the trace has an accumulation point in $[z_1,z_2)$ as $t\nearrow \tau_2$. Then $(Y_t)$ accumulates at $0$ as $t\nearrow\tau_2$. This can be seen by interpreting $(g_t(z)-W_t)$ as the limit of the probability divided by $y$ that a Brownian motion started at $iy$, $y\gg 1$, exits $\H\setminus K_t$ on the boundary arc $[\gamma_t,z]$. Consider a time where the trace is near an accumulation point in $[z_1,z_2)$
In order to exit on $[\gamma_t,z_1]$, the Brownian motion has to get near $z_2$, and then move through a strait where the trace accumulates; this conditional probability is controlled by Beurling's estimate.

The following lemma gives conditions under which the exit point of an $\SLE_\kappa(\underline\rho)$ process (first disconnection of a marked point) can be located (at least on a segment). The statement is in terms of accumulation points, which will be enough for our purposes. It is likely that a stronger statement (in terms of limits) holds. 

\begin{Lem}\label{rhohit}
Consider an $\SLE_\kappa(\underline\rho)$ in $(\H,0,\infty,z_1,\dots,z_n)$, $0<z_1<\cdots<z_n$. Let $\overline\rho_k=\rho_1+\cdots+\rho_k$, $\overline\rho_n=\kappa-6$.
\begin{enumerate}
\item Assume that for some $k$, $\overline\rho_i\geq\frac\kappa 2-2$ for $i<k$ and $\overline\rho_i\leq\frac\kappa 2-4$ for $k\leq i<n$. 
Then a.s. as $t\nearrow\tau_1$, $\gamma_t$ accumulates at $z_k$ and at no other point in $[z_1,z_n]$.
\item Assume that for some $k$, $\overline\rho_i\geq\frac\kappa 2-2$ for $i<k$; $\overline\rho_k\in(\frac\kappa 2-4,\frac\kappa 2-2)$; and $\overline\rho_{i}\leq\frac\kappa 2-4$ for $k<i<n$. 
Then a.s. as $t\nearrow\tau_1$, $\gamma_t$ accumulates at a point in $[z_k,z_{k+1}]$ and at no point in $[z_1,z_n]\setminus [z_k,z_{k+1}]$.
\end{enumerate}
\end{Lem} 

\begin{proof}
While the results are fairly intuitive from a simple Bessel dimension count, complete arguments are a bit involved.

1. a) Case $k=n=2$. By a change of coordinates, one can send $z_2$ to $\infty$. Then the $\SLE$ is defined for all times, $\tau_1=\tau_2=\infty$. This implies that the trace is unbounded (accumulates at $z_2=\infty$). Moreover, for any $z_3\in(z_1,z_2)$, the (time changed) diffusion $Y_t=(g_t(z)-W_t)/(g_t(z_1)-W_t)$ goes to 1 as $t\nearrow\tau_1=\infty$, by a study of its scale function. In particular, it does not accumulate at 0; hence the trace does not accumulate in $[z_1,z_3)$. So the only point of accumulation of the trace in $[z_1,z_2]$ is $z_2$.\\

b) Case $k=n\geq 2$. Again, we change coordinates so that $z_n=\infty$. Let $\rho_1=\frac\kappa 2-2+\rho'_1$, $\rho_i=\rho'_i-\rho'_{i-1}$. By assumption, $\rho'_i\geq 0$, $i<n$. Consider the SDE (notations as in Lemma \ref{rhodens}):
$$dZ^1_t=\frac 2{Z^1_t}dt-\sqrt\kappa dB_t+\sum_{i=1}^{n-1}\frac{\rho_i}{Z^i_t}dt=
-\sqrt\kappa dB_t+\frac\kappa 2.\frac{dt}{Z^1_t}+\sum_{i=1}^{n-1}\rho'_i\frac{Z^{i+1}_t-Z^i_t}{Z^i_tZ^{i+1}_t}dt$$
the last sum being nonnegative. By a stochastic domination argument (comparison with a Bessel process, $\delta=2$), this shows that $\tau_1=\infty$. Hence the process is defined for all times and the trace is unbounded.

Next we prove that there is no point of accumulation of the trace in $[z_1,z_{n-1}]$. Take a small neighbourhood $U$ of $[z_1,z_{n-1}]$. Let $\sigma_n$ the first time the trace goes at distance $n$ (an a.s. finite stopping time). Let $U_n$ be the connected component of 1 in $(g_{\sigma_n}(U)-W_t)/(g_{\sigma_n}(z_{n-1})-W_t)$. By harmonic measure estimates, it is easy to see that $U_n$ is contained in an arbitrarily small neighbourhood of $1$ as $n\rightarrow\infty$, say $D(1,\eps_n)$. By a), with probability $1-o(1)$, an $\SLE_\kappa(\overline\rho_{n-1})$ in $(\H,0,\infty,1)$ does not intersect $D(1,\sqrt{\eps_n})$. On the other hand, the density of the $\SLE_\kappa(\underline\rho)$ starting with all marked points in $D(1,\eps_n)$ w.r.t. to $\SLE_\kappa(\overline\rho_{n-1})$ in $(\H,0,\infty,1)$ is $1+o(1)$ on the event that the trace does not intersect $D(1,\sqrt{\eps_n})$. This follows from an inspection of the densities (Lemma \ref{rhodens}) and the fact that on the event $\{\gamma\cap D(0,\sqrt{\eps_n})=\varnothing\}$, $(g'_t(z)/g'_t(1)-1)$ is small uniformly in $t$ and $z\in [1-\eps_n,1+\eps_n]$. Indeed, Brownian excursions starting from 1 and $z$ couple with high probability before exiting $D(1,\sqrt{\eps_n})$. 

This proves that with probability $1-o(1)$, the original $\SLE$ does not return to $U$ after $\sigma_n$. Notice that one can insert a point $z'_{n}$ between $z_{n-1}$ and $z_n$ with $\rho_n=0$ and the result still applies. This shows that there is no point of accumulation in $[z_1,z_n)$.

c) Case $n=3$, $k=2$. We prove that the trace does not accumulate at $z_3$ (similarly, at $z_1$). By sending $z_3$ at infinity, it is easy to see that the half-plane capacity of the hull at $\tau$ seen from $z_1$ is finite a.s. (one can even compute its Laplace transform). We have to rule out that the hull is unbounded while having finite half-plane capacity. It is enough to prove that the driving process stopped at $\tau$ stays bounded. Since $Z^1_t=g_t(z_1)-W_t$ goes to zero as $t\nearrow\tau$, it is enough to prove that $\int^\tau\frac{ds}{Z^1_s}$ is finite. Consider the SDE for $Z^1_t$:
$$dZ^1_t=-\sqrt\kappa dB_t+\frac{\rho_1+\rho_2+2+\eps_t}{Z_1^t}$$
where $\eps_t=\rho_2(1-Z^1_t/Z^2_t)$; $\eps_t$ goes to zero as $t\nearrow\tau$ (by studying the time changed diffusion $(Z^1_t/Z^2_t)$). 
One can proceed with a comparison with Bessel processes. On the event $\{\eps_t\in[0,\eps],t\geq t_0\}$, for $t\geq t_0$, $Z^1_t$ is between a Bessel$(\delta-\eps)$ and a Bessel$(\delta)$, $\delta=1+2\frac{\rho_1+\rho_2+2}\kappa\leq 2-\frac 8\kappa$ (both hit zero in finite time). Let $t_1$ be the first time the ratio of the two bounding Bessel processes $X^-$ and $X^+$ is 2; restart them at $t_1$ from the same position $Z^1_{t_1}$, and define inductively $t_i$, $i>1$. 
One can think of restarting the majorizing Bessel process at a lower level at $t_1$ as waiting for the Bessel to reach level $Z^1_{t_1}$. This proves that $\int_{t_0}^\tau\frac{ds}{Z^1_{s}}\leq 2\int_{t_0}^{\tau_0}\frac{ds}{X^+_{s}}$, which is finite.

d) General case. Send $z_k$ to infinity by a change of coordinate. The conditions on the $\rho_i$'s are rephrased as: 
\begin{align*}
\rho_1,\rho_1+\rho_2,\dots,\rho_1+\cdots+\rho_{k-1}&\geq\frac\kappa 2-2\\
\rho_n,\rho_n+\rho_{n-1},\dots,\rho_n+\cdots+\rho_{k+1}&\geq\frac\kappa 2-2
\end{align*}

Hence the situation to the left and to the right of 0 are identical. It is easy to see from b) that the trace is defined for all times and is unbounded. Let $\sigma_n$ be the time of first exit of $D(0,n)$ by $\gamma$. Rescale the process so that $g_t(z_1)$ (resp. $g_t(z_n)$, $W_t$) is sent to $-1$ (resp. $1$, $w_t$) at $t=\sigma_n$. If $w_t$ is away from $\pm 1$, one can reason as in b) from the result of c). If $w$ is close to 1, say, one can rescale by sending $w$ to $0$ (and 1) is fixed. The resulting process has density very close to 1 with a process of type b) as long as $w_t$ stays close to 1. When $w_t$ separates from 1, one can apply c) with a density argument.

2. a) Case $n=2,k=1$. It is easily seen by sending $z_2$ (or $z_1$) to infinity that the trace is defined for a finite time. 
Reasoning as in 1c) shows that the trace is bounded. Hence it accumulates somewhere in $(z_1,z_2)$, but not at $z_2$ (and by symmetry $z_1$).

b) General case. Send $z_k$ to infinity (so that $z_{k+1}<0$). A stochastic domination argument as in 1b) shows that the driving process is dominated by the one corresponding to $z_2,\dots,z_{k-1}$ being sent to infinity while $z_{k+2},\dots,z_n$ are sent to $z_{k+1}$. It is easily seen that for that process, $\tau_{k+1}<\tau_1$. Consequently, this is also the case for the original process, viz. the trace accumulates on $[z_k,z_n]$ without accumulating on $[z_1,z_k)$. 
As in 1d), the situation is symmetric, so there is also no accumulation on $(z_{k+1},z_n]$.
\end{proof}

\bibliographystyle{abbrv}
\bibliography{biblio}

-----------------------

\noindent The University of Chicago\\
Department of Mathematics\\
5734 S. University Avenue\\
Chicago IL 60637\\

\end{document}